\documentclass{amsproc}
\usepackage{amsmath}
\usepackage{enumerate}
\usepackage{amsmath,amsthm,amscd,amssymb}

\usepackage{latexsym}
\usepackage{upref}
\usepackage{verbatim}

\usepackage[mathscr]{eucal}
\usepackage{dsfont}

\usepackage{graphicx}

\usepackage{hhline}

\usepackage[OT2,OT1]{fontenc}
\newcommand\cyr
{
\renewcommand\rmdefault{wncyr}
\renewcommand\sfdefault{wncyss}
\renewcommand\encodingdefault{OT2}
\normalfont
\selectfont
}
\DeclareTextFontCommand{\textcyr}{\cyr}
\def\cprime{\char"7E }

\newtheorem{theorem}{Theorem}

\newtheorem{lemma}[theorem]{Lemma}
\newtheorem{definition}[theorem]{Definition}
\newtheorem{remark}[theorem]{Remark}
\newtheorem{hypothesis}[theorem]{Hypothesis}

\chardef\bslash=`\\ 

\newcommand{\wh}{\widehat}

\newcommand{\dA}{{\dot A}}

\newcommand{\bbR}{{\mathbb{R}}}

\newcommand{\bbC}{{\mathbb{C}}}

\newcommand{\ti}{\tilde  }

\newcommand{\dom}{\text{\rm{Dom}}}

\newcommand{\calD}{{\mathcal D}}

\newcommand{\calH}{{\mathcal H}}

\newcommand{\calR}{{\mathcal R}}

\newcommand{\calS}{{\mathcal S}}

\newcommand{\mM}{\mathfrak M}

\newcommand{\whA}{T}

\renewcommand{\Im}{\text{\rm Im}}

\def\sM{{\mathfrak M}}   \def\sN{{\mathfrak N}}

\def\bA{{\mathbb A}}      \def\dC{{\mathbb C}}

      \def\dR{{\mathbb R}}

   \def\cH{{\mathcal H}}

\def\RE{{\rm Re\,}}
\def\Ker{{\rm Ker\,}}

\def\wh{\hat}

\def\uphar{{\upharpoonright\,}}

\DeclareMathOperator{\IM}{Im}

\newcommand{\eval}[2][\right]{\relax
  \ifx#1\right\relax \left.\fi#2#1\rvert}

\begin{document}

\title{On the c-entropy of L-systems with Schr\"odinger  operator}

\author{S. Belyi}
\address{Department of Mathematics\\ Troy University\\
Troy, AL 36082, USA\\
}
\curraddr{}
\email{sbelyi@troy.edu}


\author[Makarov]{K. A. Makarov}
\address{Department of Mathematics\\
 University of Missouri\\
  Columbia, MO 63211, USA}
\email{makarovk@missouri.edu}

\thanks{The second author was partially supported by the Simons collaboration
grant  00061759  while preparing this paper.}

\author{E. Tsekanovskii}
\address{Department of Mathematics\\ Niagara University, New York
14109\\ USA}
\email{tsekanov@niagara.edu}


\date{DD/MM/2004}

\dedicatory{In loving memory of Moshe Liv\u sic}


\begin{abstract}
We study L-systems whose main operators are  extensions of one-dimensional half-line Schr\"odinger operators with deficiency indices $(1, 1)$, the Schr\"odinger L-systems. 
Introducing new concepts of an c-entropy and dissipation coefficient for an L-system  we discuss the following  dual problems: describe Schr\"odinger L-systems (1) with a given c-entropy and minimal  dissipation coefficient, and (2) with a  given dissipation coefficient and maximal c-entropy. Also, we analyze  in detail the  dual c-entropy problems for Schr\"odinger L-systems with sectorial and extremal main operators. 

 \end{abstract}

\maketitle



\textit{This paper is dedicated to the memory of Moshe Liv\u sic, a remarkable Human Being and Mathematician. His pioneering research in the theory of non-selfadjoint operators  and system theory  \cite{Lv2} has made writing this paper possible. The following lines from a Russian poet Nikolaj Gumilev (Prophets, 1905) describe the life and the light of scientific accomplishments of Liv\u sic really well}

\begin{center}
{\cyr \qquad\textsf{\ \ \  I nyne est{\cprime}  eshche proroki,}}

{\cyr \textsf{{\ Hotya upali altari,}}}

{\cyr \qquad\textsf{Ih ochi yasny i gluboki,}}

{\cyr \qquad\textsf{\ \ \ Gryadushchim plamenem zari.}}
\end{center}

\section{Introduction}\label{s1}

In the current paper we set focus on L-systems  whose main operators are  extensions of the minimal  one-dimensional Schr\"odinger operators with deficiency indices $(1, 1)$ on the half-line. 
We refer to such  L-systems as \textit{Schr\"odinger L-systems}. As a main highlight of the current paper we introduce a new concept of an L-system \textit{coupling entropy} (or \textit{c-entropy}), which is an additive quantity with respect to the coupling of L-systems (see  \cite{BMkT-2} for the concept of coupling).
 We relate the c-entropy of a Schr\"odinger L-system to the dissipation coefficient of its main operator $T_h$ and pose the following  \textit{dual c-entropy  problems}. The first problem is to describe a Schr\"odinger L-system with a given c-entropy and minimal  dissipation coefficient, while the second one is to construct a Schr\"odinger L-system that has a given dissipation coefficient and maximal c-entropy.
 We solve these dual problems for several classes of Schr\"odinger L-systems, in particular, for the Schr\"odinger L-systems whose impedance functions belong to one of the generalized Donoghue classes  $\sM_\kappa$ and $\sM_\kappa^{-1}$ introduced in \cite{BMkT}.  As an auxiliary result,  we obtain  a criteria  (in terms of the boundary value of the main operators) for the impedance functions  to belong to one of (generalized) Donog\-hue classes $\sM$, $\sM_\kappa$, and $\sM_\kappa^{-1}$ and then  solve the dual c-entropy problems for     Schr\"odinger L-systems with extremal and $\beta$-sectorial main operators.

The paper is organized as follows. The formal definitions of general and Schr\"odinger L-systems as well as corresponding function classes are presented in Sections \ref{s2},  \ref{s3} and \ref{s4}. We capitalize on the fact that class of Schr\"odinger L-systems $\Theta_{\mu,h}$ forms a two-parametric family whose members are uniquely defined by a real-valued parameter $\mu$ and a complex boundary value $h$, ($\IM h>0$) of the main operator. Here the parameter $\mu\in\Bbb R\cup \{\infty\}$ uniquely defines (for a given $h$) the state-space operator $\bA_{\mu,h}$ of the L-system $\Theta_{\mu,h}$,  thus fixing  $\Theta_{\mu,h}$ in a unique way.

 In Section \ref{s4} we establish a connection between the absolute value  of the von Neumann parameter of the main operator $T_h$ and its boundary value parameter $h$ and put forward the definition of the dissipation coefficient.  Also, we obtain a  criteria for the impedance functions of L-systems with Schr\"odinger operator to fall into one of the (generalized) Donog\-hue classes $\sM$, $\sM_\kappa$, and $\sM_\kappa^{-1}$ (see Theorems \ref{t-6}-\ref{t-7-1}). It is worth mentioning that  Theorem \ref{t-6} describes an entire family of Schr\"odinger L-systems whose impedance functions belong to the Donog\-hue class $\sM$, while  Theorems \ref{t-7}-\ref{t-7-1} deal with  explicitly  defined Schr\"odinger L-systems whose impedance functions are members of the generalized Donoghue classes $\sM_\kappa$ and $\sM_\kappa^{-1}$, respectively.

In Sections \ref{s5} and \ref{s6} we discuss  an analogue  of the Phillips-Kato extension problem for a non-negative Schr\"odinger operator  with deficiency indices $(1, 1)$ in the Hilbert space $\calH=L^2(\ell, \infty)$. Recall that in  general the Philips-Kato extension problem concerns the existence and description of all maximal accretive and/or sectorial non-self-adjoint extensions $A$ of a non-negative symmetric operator $\dA$ such that $\dA\subset A\subset\dA^*$.

In Section \ref{s7} we introduce the concept of an L-system c-entropy and relate it to   L-system's dissipation coefficient introduced  in Section \ref{s4}. Two following dual problems  associated with the c-entropy of  Schr\"odinger L-systems are considered:
\begin{itemize}
      \item \textit{Give a description of an L-system with the Schr\"odinger  dissipative main operator that has a given c-entropy and the  \textit{minimal } dissipation coefficient.}
      \item \textit{Describe an L-system with Schr\"odinger main operator that has a given dissipation coefficient and the \textit{maximal} c-entropy.}
\end{itemize}
In Sections \ref{s7} and \ref{s8} we respectively present the solutions to both problems posed for the Schr\"odinger L-systems whose impedance functions belong to one of the generalized Donoghue classes  $\sM_\kappa$ and $\sM_\kappa^{-1}$.

In Section \ref{s9} we solve the dual c-entropy problems for the  classes of  Schr\"odinger L-systems with extremal and $\beta$-sectorial main operators. Dealing with these two cases we do  not require  that the impedance functions of the Schr\"odinger L-systems in question   belong to one of the generalized Donoghue classes. As   a result, we present the solution to dual c-entropy problems for the entire one-parametric family of Schr\"odinger L-systems $\Theta_{\mu,h}$.
\   We also treat a combined case of extremal and $\beta$-sectorial main operators to see when the  corresponding  L-system has a maximal c-entropy. It turns out that the Schr\"odinger L-system with accretive (either $\beta$-sectorial ($\beta\in(0,\pi/2)$ or extremal) main operator $T_h$  attains the maximum c-entropy when the main operator is extremal accretive. Moreover, in the case when both main and state-space operators are extremal,  the quasi-kernel in the corresponding Schr\"odinger L-system coincides with  the Krein-von Neumann extension of the underlying symmetric operator.
In addition, we find the conditions when the impedance function of the Schr\"odinger L-system with the main extremal operator that has a maximum c-entropy belongs to the generalized Donoghue classes $\sM_\kappa$ or $\sM_\kappa^{-1}$, respectively.

We  conclude the paper with providing  examples that illustrate   main results.

\section{Preliminaries}\label{s2}

For a pair of Hilbert spaces $\calH_1$, $\calH_2$ we denote by
$[\calH_1,\calH_2]$ the set of all bounded linear operators from
$\calH_1$ to $\calH_2$. Let $\dA$ be a closed, densely defined,
symmetric operator in a Hilbert space $\calH$ with inner product
$(f,g),f,g\in\calH$. Any non-symmetric operator $T$ in $\cH$ such that
\[
\dA\subset T\subset\dA^*
\]
is called a \textit{quasi-self-adjoint extension} of $\dA$.

 Consider the rigged Hilbert space (see \cite{ABT}, \cite{Ber})
$\calH_+\subset\calH\subset\calH_- ,$ where $\calH_+ =\dom(\dA^*)$ and
\begin{equation}\label{108}
(f,g)_+ =(f,g)+(\dA^* f, \dA^*g),\;\;f,g \in \dom(A^*).
\end{equation}
Let $\calR$ be the \textit{\textrm{Riesz-Berezansky   operator}} $\calR$ (see  \cite{ABT}, \cite{Ber}) which maps $\mathcal H_-$ onto $\mathcal H_+$ such
 that   $(f,g)=(f,\calR g)_+$ ($\forall f\in\calH_+$, $g\in\calH_-$) and
 $\|\calR g\|_+=\| g\|_-$.
 Note that
identifying the space conjugate to $\calH_\pm$ with $\calH_\mp$, we
get that if $\bA\in[\calH_+,\calH_-]$, then
$\bA^*\in[\calH_+,\calH_-].$
An operator $\bA\in[\calH_+,\calH_-]$ is called a \textit{self-adjoint
bi-extension} of a symmetric operator $\dA$ if $\bA=\bA^*$ and $\bA
\supset \dA$.
Let $\bA$ be a self-adjoint
bi-extension of $\dA$ and let the operator $\hat A$ in $\cH$ be defined as follows:
\[
\dom(\hat A)=\{f\in\cH_+:\bA f\in\cH\}, \quad \hat A=\bA\uphar\dom(\hat A).
\]
The operator $\hat A$ is called a \textit{quasi-kernel} of a self-adjoint bi-extension $\bA$ (see \cite[Section 2.1]{ABT}, \cite{TSh1}).
According to the von Neumann Theorem (see \cite[Theorem 1.3.1]{ABT}) the domain of $\wh A$, a self-adjoint extension of $\dA$,  can be expressed as
\begin{equation}\label{DOMHAT}
\dom(\hat A)=\dom(\dA)\oplus(I+U)\sN_{i},
\end{equation}
where von Neumann's parameter $U$ is both a $(\cdot)$-isometric as well as $(+)$-isometric operator from $\sN_i$ into $\sN_{-i}$ and $$\sN_{\pm i}=\Ker (\dA^*\mp i I)$$ are the deficiency subspaces of $\dA$.
 A self-adjoint bi-extension $\bA$ of a symmetric operator $\dA$ is called \textit{t-self-adjoint} (see \cite[Definition 4.3.1]{ABT}) if its quasi-kernel $\hat A$ is a self-adjoint operator in $\calH$.
An operator $\bA\in[\calH_+,\calH_-]$  is called a \textit{quasi-self-adjoint bi-extension} of an operator $T$ if $\bA\supset T\supset \dA$ and $\bA^*\supset T^*\supset\dA.$

We will be mostly interested in the following type of quasi-self-adjoint bi-extensions.
Let $T$ be a quasi-self-adjoint extension of $\dA$ with nonempty resolvent set $\rho(T)$. A quasi-self-adjoint bi-extension $\bA$ of an operator $T$ is called (see \cite[Definition 3.3.5]{ABT}) a \textit{($*$)-extension } of $T$ if $\RE\bA$ is a
t-self-adjoint bi-extension of $\dA$.
In what follows we assume that $\dA$ has deficiency indices $(1,1)$. In this case it is known \cite{ABT} that every  quasi-self-adjoint extension $T$ of $\dA$  admits $(*)$-extensions.
The description of all $(*)$-extensions via the Riesz-Berezansky   operator $\calR$ can be found in \cite[Section 4.3]{ABT}.

Recall that a linear operator $T$ in a Hilbert space $\calH$ is called \textbf{accretive} \cite{Ka} if $\RE(Tf,f)\ge 0$ for all $f\in \dom(T)$.  We call an accretive operator $T$
\textbf{$\alpha$-sectorial} \cite{Ka} if there exists a value of $\alpha\in(0,\pi/2)$ such that
\begin{equation}\label{e8-29}
    (\cot\alpha) |\IM(Tf,f)|\le\,\RE(Tf,f),\qquad f\in\dom(T).
\end{equation}
We say that the angle of sectoriality $\alpha$ is \textbf{exact} for an $\alpha$-sectorial
operator $T$ if $$\tan\alpha=\sup_{f\in\dom(T)}\frac{|\IM(Tf,f)|}{\RE(Tf,f)}.$$
An accretive operator is called \textbf{extremal accretive} if it is not $\alpha$-sectorial for any $\alpha\in(0,\pi/2)$. In what follows, when we say that an accretive operator is $\alpha$-sectorial, we mean that $\alpha$ is its exact angle of sectoriality unless otherwise is specified.

The following definition is a ``lite" version of the definition of L-system given for a scattering L-system with
 one-dimensional input-output space. It is tailored for the case when the symmetric operator of an L-system has deficiency indices $(1,1)$. The general definition of an L-system can be found in \cite[Definition 6.3.4]{ABT}.
\begin{definition} 
 An
array
\begin{equation}\label{e6-3-2}
\Theta= \begin{pmatrix} \bA&K&\ 1\cr \calH_+ \subset \calH \subset
\calH_-& &\dC\cr \end{pmatrix}
\end{equation}
 is called an \textbf{{L-system}}   if:
\begin{enumerate}
\item[(1)] {$T$ is a dissipative ($Im(Tf,f)\ge0$, $f\in\dom(T)$) quasi-self-adjoint extension of a symmetric operator $\dA$ with deficiency indices $(1,1)$};
\item[(2)] {$\mathbb  A$ is a   ($\ast $)-extension of  $T$};
\item[(3)] $\IM\bA= KK^*$, where $K\in [\dC,\calH_-]$ and $K^*\in [\calH_+,\dC]$.
\end{enumerate}
\end{definition}
  Operators $T$ and $\bA$ are called the \textit{main and state-space operators respectively} of the system $\Theta$, and $K$ is  the \textit{channel operator}.
It is easy to see that the operator $\bA$ of the system  \eqref{e6-3-2}  can be chosen in such a way  that $\IM\bA=(\cdot,\chi)\chi$, $\chi\in\calH_-$ and $K c=c\cdot\chi$, $c\in\dC$.
  A system $\Theta$ in \eqref{e6-3-2} is called \textit{minimal} if the operator $\dA$ is a prime operator in $\calH$, i.e., there exists no non-trivial reducing invariant subspace of $\calH$ on which it induces a self-adjoint operator. Notice that minimal L-systems of the form \eqref{e6-3-2} with  one-dimensional input-output space were also considered in \cite{BMkT}.

We  associate with an L-system $\Theta$ the  function
\begin{equation}\label{e6-3-3}
W_\Theta (z)=I-2iK^\ast (\mathbb  A-zI)^{-1}K,\quad z\in \rho (T),
\end{equation}
 which is called the \textbf{transfer  function} of the L-system $\Theta$. We also consider another  function related to an L-system $ \Theta $ called the \textbf{impedance function} and given by the formula
\begin{equation}\label{e6-3-5}
V_\Theta (z) = K^\ast (\RE\bA - zI)^{-1} K.
\end{equation}

 The transfer function $W_\Theta (z)$ of the L-system $\Theta $ and function $V_\Theta (z)$ of the form (\ref{e6-3-5}) are connected by the following relations valid for $\IM z\ne0$, $z\in\rho(T)$,
\begin{equation*}\label{e6-3-6}
\begin{aligned}
V_\Theta (z) &= i [W_\Theta (z) + I]^{-1} [W_\Theta (z) - I],\\
W_\Theta(z)&=(I+iV_\Theta(z))^{-1}(I-iV_\Theta(z)).
\end{aligned}
\end{equation*}
 The class of all Herglotz-Nevanlinna functions, that can be realized as impedance functions of L-systems, and connections with Weyl-Titchmarsh functions can be found in \cite{ABT}, \cite{BMkT},  \cite{DMTs},  \cite{GT} and references therein.
In particular it is shown there that any impedance function $V_\Theta(z)$ admits the  integral representation
\begin{equation}\label{e-60-nu}
V_\Theta(z)=Q+\int_\bbR \left(\frac{1}{\lambda-z}-\frac{\lambda}{1+\lambda^2}\right)d\sigma,
\end{equation}
where $Q$ is a real number and $\sigma$ is an  infinite Borel measure   such that
$$
\int_\bbR\frac{d\sigma(\lambda)}{1+\lambda^2}<\infty.
$$

\section{Donoghue classes and L-systems with one-dimensional input-output}\label{s3}

Suppose that $\dA$ is a closed prime 
densely defined symmetric operator with deficiency indices $(1,1)$. Assume also that $T \ne T^*$ is a  maximal dissipative extension of $\dot A$,
$$\Im(T f,f)\ge 0, \quad f\in \dom(T ).$$
Since $\dot A$ is symmetric, its dissipative extension $T$ is automatically quasi-self-adjoint \cite{ABT},
that  is,
$$
\dot A \subset T \subset \dA^*,
$$
and hence, (see \cite{BMkT})
\begin{equation}\label{parpar}
g_+-\kappa g_-\in \dom(T)\quad \text{for some } |\kappa|<1,
\end{equation}
where $g_\pm\in\sN_{\pm i}=\Ker (\dA^*\mp i I)$ and  $\|g_\pm\|=1$.
{Throughout this paper  $\kappa$ will be referred to as the \textbf{ von Neumann  parameter} of the operator $T$.}
The next lemma contains a characterization of sectorial operators in terms of the modulus of von Neumann's parameter.

\begin{lemma}\label{l-2}
If $T$ is an $\alpha$-sectorial operator with $\alpha\in(0,\pi/2)$, then its von Neumann's  parameter $\kappa$ cannot equal  zero.
\end{lemma}
\begin{proof}
Assume the contrary, let $T$ be an $\alpha$-sectorial operator in a Hilbert space $\calH$ with $\alpha\in(0,\pi/2)$ and  the von Neumann parameter $\kappa=0$. Then \eqref{parpar} implies (see \cite{ABT}) that there exists a non-zero vector $x\in\dom(T)$ such that $Tx=ix$, $x\ne0$. Moreover,
$$
(Tx,x)=(ix,x)=i\|x\|^2 \quad\textrm{ and }\quad (x,Tx)=(x,ix)=-i\|x\|^2.
$$
Then
$$
\RE(Tx,x)=\frac{(Tx,x)+(x,Tx)}{2}=0
$$
and
$$
 \IM(Tx,x)=\frac{(Tx,x)-(x,Tx)}{2i}=\frac{i\|x\|^2+i\|x\|^2}{2i}=\|x\|^2.
$$
But  $T$ is $\alpha$-sectorial and hence \eqref{e8-29} takes place implying
$$
0\le(\cot\alpha) |\IM(Tx,x)|\le\,\RE(Tx,x)=0.
$$
This yields $0\le(\cot\alpha)\|x\|^2\le0$ or $(\cot\alpha)\|x\|^2=0$. Since neither $\cot\alpha$ nor $\|x\|$ can equal zero, we reached a contradiction. Therefore, $\kappa\ne0$.
\end{proof}

Recall that  Donoghue \cite{D}  introduced a concept of the Herglotz-Nevanlinna function $M_{(\dot A, A)}(z)$ associated with the pair $(\dot A, A)$ by
$$M_{(\dot A, A)}(z)=
((Az+I)(A-zI)^{-1}g_+,g_+), \quad z\in \bbC_+,
$$
$$g_+\in \Ker( \dA^*-iI),\quad \|g_+\|=1,$$
where $\dot A $ is a symmetric operator with deficiency indices $(1,1)$, and $A$ is its self-adjoint extension.
Let $\sN$ (see \cite{BMkT-3}) be a class of all Herglotz-Nevanlinna functions $M(z)$ that admit the representation
\begin{equation}\label{hernev-0}
M(z)=\int_\bbR \left(\frac{1}{\lambda-z}-\frac{\lambda}{1+\lambda^2}\right)d\sigma,
\end{equation}
where $\sigma$ is an  infinite Borel measure with
$$
\int_\bbR\frac{d\sigma(\lambda)}{1+\lambda^2}<\infty.
$$
 Following our earlier developments in \cite{BMkT}, \cite{BMkT-3}, \cite{MT10}, \cite{MT2021}   denote by $\sM$ the \textbf{Donoghue class} of all analytic mappings $M(z)$ from $\bbC_+$ into itself  that admits the representation
 \eqref{hernev-0}    and has a property
$$
\int_\bbR\frac{d\sigma(\lambda)}{1+\lambda^2}=1\,,\quad\text{equivalently,}\quad M(i)=i.
$$

It is known  \cite{D}, \cite{GMT97},  \cite{GT}, \cite {MT-S} that $M(z)\in \mM$ if and only if $M(z)$ can be realized  as the Weyl-Titchmarsh function $M_{(\dot A, A)}(z)$ associated with the pair $(\dot A, A)$. Furthermore, we say (see \cite{BMkT}) that a function $M(z)\in\sN$  belongs to the \textbf{generalized Donoghue class} $\sM_\kappa$, ($0\le\kappa<1$) if  in the representation \eqref{hernev-0}
\begin{equation}\label{e-38-kap}
\int_\bbR\frac{d\sigma(\lambda)}{1+\lambda^2}=\frac{1-\kappa}{1+\kappa}\,,\quad\text{equivalently,}\quad M(i)=i\,\frac{1-\kappa}{1+\kappa}.
\end{equation}
Similarly  (see \cite{BMkT-2}), a function $M(z)\in\sN$  belongs to the  \textbf{generalized Donoghue class} $\sM_\kappa^{-1}$ if in the representation \eqref{hernev-0}
\begin{equation}\label{e-39-kap}
\int_\bbR\frac{d\sigma(\lambda)}{1+\lambda^2}=\frac{1+\kappa}{1-\kappa}\,,\quad\text{equivalently,}\quad M(i)=i\,\frac{1+\kappa}{1-\kappa}.
\end{equation}
Clearly, when $\kappa=0$ the generalized Donoghue classes $\sM_\kappa$ and $\sM_\kappa^{-1}$ coincide with the Donoghue class $\sM$, that is $\sM_0=\sM_0^{-1}=\sM$. If $M(z)$ is an arbitrary function from $\sN$ with a normalization condition
\begin{equation}\label{e-66-L}
\int_\bbR\frac{d\sigma(\lambda)}{1+\lambda^2}=a,
\end{equation}
for some $a>0$, then it is easy to see that $M(z)\in\sM$ if and only if $a=1$. Also, if $a<1$, then $M(z)\in \sM_\kappa$ with
\begin{equation}\label{e-45-kappa-1}
\kappa=\frac{1-a}{1+a},
\end{equation}
and if $a>1$, then $M(z)\in \sM_\kappa^{-1}$ with
\begin{equation}\label{e-45-kappa-2}
\kappa=\frac{a-1}{1+a}.
 \end{equation}

\begin{hypothesis}\label{setup} Suppose that $\whA \ne\whA^*$  is  a maximal dissipative extension of  a symmetric operator $\dot A$  with deficiency indices $(1,1)$. Assume, in addition, that $A$ is a  self-adjoint extension of $\dot A$. Suppose,  that the deficiency elements $g_\pm\in \Ker (\dA^*\mp iI)$ are normalized, $\|g_\pm\|=1$, and chosen in such a way that
\begin{equation}\label{ddoomm14}g_+- g_-\in \dom ( A)\,\,\,\text{and}\,\,\,
g_+-\kappa g_-\in \dom (\whA )\,\,\,\text{for some }
\,\,\,|\kappa|<1.
\end{equation}
\end{hypothesis}
It is known \cite{MT-S} that if $\kappa=0$, then  quasi-self-adjoint extension $\whA $ coincides with the restriction of the adjoint operator $\dot A^*$ on
$$
\dom(\whA )=\dom(\dot A)\dot + \Ker (\dA^*-iI).
$$
Similar to Hypothesis \ref{setup} it is convenient to adopt ``anti-Hypothesis".
\begin{hypothesis}\label{setup-1} Suppose
that $\whA \ne\whA^*$  is  a maximal
dissipative extension of  a symmetric operator $\dot A$
 with deficiency indices $(1,1)$. Assume, in addition, that
$A$ is a  self-adjoint extension of $\dot A$. Suppose,  that the
deficiency elements $g_\pm\in \Ker (\dA^*\mp iI)$ are
normalized, $\|g_\pm\|=1$, and chosen in such a way that
\begin{equation}\label{ddoomm14-1}g_++ g_-\in \dom ( A)\,\,\,\text{and}\,\,\,
g_+-\kappa g_-\in \dom (\whA )\,\,\,\text{for some }
\,\,\,|\kappa|<1.
\end{equation}
\end{hypothesis}
\noindent
\begin{remark}\label{r-12}
Without loss of generality, in what follows we assume that $\kappa$ is real and $0\le\kappa<1$: if $\kappa=|\kappa|e^{i\theta}$,
change (the basis) $g_-$ to $e^{i\theta}g_-$ in the deficiency subspace  $\Ker (\dA^*+ i I)$.
\end{remark}
This remark means the following: let
\begin{equation}\label{e-62}
\Theta= \begin{pmatrix} \bA&K&\ 1\cr \calH_+ \subset \calH \subset
\calH_-& &\dC\cr \end{pmatrix}
\end{equation}
be a minimal L-system  with one-dimensional input-output space $\dC$. 
If the main operator $T$ of $\Theta$ is parameterized with a \textit{complex} von Neumann's parameter $\kappa$ that corresponds to a chosen normalized pair of deficiency vectors $g_+$ and $g_-$, then we can change the deficiency basis as explained in Remark \ref{r-12} and represent $T$ using real value  $|\kappa|$ with respect to the new deficiency basis. This procedure will change the parameter $U$ of the quasi-kernel $\hat A$ of $\RE\bA$ in  \eqref{DOMHAT} and ultimately the way $\bA$ is described.  \textit{Thus, for the remainder of this paper (unless otherwise is specified) we will consider L-systems  \eqref{e-62} such that $\kappa$ is real and $0\le\kappa<1$.}

\begin{definition}
We say that an L-system $\Theta$ of the form \eqref{e-62}  \textit{satisfies Hypothesis} \ref{setup} (or \ref{setup-1}) if its main operator $T$ and the quasi-kernel $\hat A$ of $\RE\bA$ satisfy the conditions of Hypothesis \ref{setup} (or \ref{setup-1}) for a fixed set of deficiency vectors of the symmetric operator $\dA$.
\end{definition}

Let $\Theta$ be a minimal L-system of the form \eqref{e-62} that satisfies  Hypothesis \ref{setup}. It is shown in \cite{BMkT} that  the impedance function $V_\Theta(z)$ can be represented as
\begin{equation}\label{e-imp-m}
    V_{\Theta}(z)=\left(\frac{1-\kappa}{1+\kappa}\right)V_{\Theta_0}(z),
\end{equation}
where $V_{\Theta_0}(z)$ is  the impedance function of an L-system $\Theta_0$ with the same set of conditions but with $\kappa_0=0$, where $\kappa_0$ is the von Neumann parameter of the main operator $T_0$ of $\Theta_0$.

Let $\Theta_1$ and $\Theta_2$ be two minimal L-system of the form \eqref{e-62}  whose components satisfy  Hypothesis \ref{setup} and Hypothesis \ref{setup-1}, respectively. Then it was proved in \cite[Lemma 5.1]{BMkT-2} that the impedance functions $V_{\Theta_1}(z)$ and $V_{\Theta_2}(z)$ admit the integral representation
\begin{equation}\label{e-60-nu-1}
V_{\Theta_{k}}(z)=\int_\bbR \left(\frac{1}{t-z}-\frac{t}{1+t^2}\right )d\sigma_{k}(t),\quad k=1,2.
\end{equation}

Now let us consider a minimal L-system $\Theta$ of the form \eqref{e-62} that satisfies  Hypothesis \ref{setup}. Let also
\begin{equation}\label{e-62-alpha}
\Theta_\alpha= \begin{pmatrix} \bA_\alpha&K_\alpha&\ 1\cr \calH_+ \subset \calH \subset
\calH_-& &\dC\cr \end{pmatrix},\quad \alpha\in[0,\pi),
\end{equation}
 be a one parametric family of L-systems such that
 \begin{equation}\label{e-63-alpha}
    W_{\Theta_\alpha}(z)=W_\Theta(z)\cdot (-e^{2i\alpha}),\quad \alpha\in[0,\pi).
 \end{equation}
The existence and structure of $\Theta_\alpha$ were described in details in \cite[Section 8.3]{ABT}. In particular, it was shown that the L-system $\Theta$ and $\Theta_\alpha$ share the same main operator $T$ and that
\begin{equation}\label{e-64-alpha}
    V_{\Theta_\alpha}(z)=\frac{\cos\alpha+(\sin\alpha) V_\Theta(z)}{\sin\alpha-(\cos\alpha) V_\Theta(z)}.
\end{equation}

Let $\Theta$ be a minimal L-system $\Theta$ of the form \eqref{e-62} that satisfies Hypothesis \ref{setup}. Also let $\Theta_{\alpha}$ be a one parametric family of L-systems given by \eqref{e-62-alpha}, \eqref{e-63-alpha}.
It was shown in \cite[Theorem 5.2]{BMkT-2} that in this case the impedance function $V_{\Theta_{\alpha}}(z)$ has an integral representation
$$
V_{\Theta_{\alpha}}(z)=\int_\bbR \left(\frac{1}{t-z}-\frac{t}{1+t^2}\right )d\sigma_{\alpha}(t)
$$
if and only if $\alpha=0$ or $\alpha=\pi/2$.

The next result describes the relationship between two L-systems of the form \eqref{e-62} that comply with different hypotheses.

Let
\begin{equation}\label{e-62-1}
\Theta_1= \begin{pmatrix} \bA_1&K_1&\ 1\cr \calH_+ \subset \calH \subset
\calH_-& &\dC\cr \end{pmatrix}
\end{equation}
be a minimal L-system whose main operator $T$ and the quasi-kernel $\hat A_1$ of $\RE\bA_1$ satisfy the conditions of Hypothesis \ref{setup} and let
\begin{equation}\label{e-62-2}
\Theta_2= \begin{pmatrix} \bA_2&K_2&\ 1\cr \calH_+ \subset \calH \subset
\calH_-& &\dC\cr \end{pmatrix}
\end{equation}
be another minimal L-system with the same  operators $\dA$ and $T$ as $\Theta_1$ but with the quasi-kernel $\hat A_2$ of $\RE\bA_2$ that satisfies  Hypothesis \ref{setup-1}. It was shown in \cite[Theorem 5.3]{BMkT-2} that
\begin{equation}\label{e-55-1}
    W_{\Theta_1}(z)=-W_{\Theta_2}(z),\quad z\in\dC_+\cap\rho(T)
\end{equation}
and
\begin{equation}\label{e-56-1}
    V_{\Theta_1}(z)=-\frac{1}{V_{\Theta_2}(z)},\quad z\in\dC_+\cap\rho(T).
\end{equation}

\section{L-systems with Schr\"odinger operator $T_h$ and their impedance functions}\label{s4}

Let $\calH=L_2(\ell,+\infty)$ and $l(y)=-y^{\prime\prime}+q(x)y$, where $q$ is a real locally summable function. Suppose that the minimal symmetric operator
\begin{equation}
\label{128}
 \left\{ \begin{array}{l}
 \dA y=-y^{\prime\prime}+q(x)y \\
 y(\ell)=y^{\prime}(\ell)=0 \\
 \end{array} \right.
\end{equation}
has deficiency indices (1,1). Let $D^*$ be the set of functions locally absolutely continuous together with their first derivatives such that $l(y) \in L_2(\ell,+\infty)$. Consider $\calH_+=\dom(\dA^*)=D^*$ with the scalar product
$$(y,z)_+=\int_{a}^{\infty}\left(y(x)\overline{z(x)}+l(y)\overline{l(z)}
\right)dx,\;\; y,\;z \in D^*.$$ Let $\calH_+ \subset L_2(\ell,+\infty) \subset \calH_-$ be the corresponding triplet of Hilbert spaces. Consider the operators (cf. \cite{PavDrog})
\begin{equation}\label{131}
 \left\{ \begin{array}{l}
 T_hy=l(y)=-y^{\prime\prime}+q(x)y \\
 hy(\ell)-y^{\prime}(\ell)=0 \\
 \end{array} \right.
           ,\quad  \left\{ \begin{array}{l}
 T^*_hy=l(y)=-y^{\prime\prime}+q(x)y \\
 \overline{h}y(\ell)-y^{\prime}(\ell)=0 \\
 \end{array} \right..
\end{equation}

Let  $\dA$ be a symmetric operator  of the form \eqref{128} with deficiency indices (1,1), generated by the differential expression $l(y)=-y^{\prime\prime}+q(x)y$. Let also $\varphi_k(x,\lambda) (k=1,2)$ be the solutions of the following Cauchy problems:
$$\left\{ \begin{array}{l}
 l(\varphi_1)=\lambda \varphi_1 \\
 \varphi_1(\ell,\lambda)=0 \\
 \varphi'_1(\ell,\lambda)=1 \\
 \end{array} \right., \qquad
\left\{ \begin{array}{l}
 l(\varphi_2)=\lambda \varphi_2 \\
 \varphi_2(\ell,\lambda)=-1 \\
 \varphi'_2(\ell,\lambda)=0 \\
 \end{array} \right.. $$
It is well known \cite{Na68} that there exists a Weyl function $m_\infty(\lambda)$  such that
$$\varphi(x,\lambda)=\varphi_2(x,\lambda)+m_\infty(\lambda)
\varphi_1(x,\lambda)$$ belongs to $L_2(\ell,+\infty)$.

Now we shall construct an L-system associated with a non-self-adjoint Schr\"odin\-ger operator $T_h$.  It  was shown in \cite{ABT}, \cite{ArTs0} that  the set of all ($*$)-extensions of the non-self-adjoint Schr\"odinger operator $T_h$ of the form \eqref{131} in $L_2(\ell,+\infty)$ can be represented as
\begin{equation}\label{137}
\begin{split}
&\bA_{\mu, h}\, y=-y^{\prime\prime}+q(x)y-\frac {1}{\mu-h}\,[y^{\prime}(\ell)-
hy(\ell)]\,[\mu \delta (x-\ell)+\delta^{\prime}(x-\ell)], \\
&\bA^*_{\mu, h}\, y=-y^{\prime\prime}+q(x)y-\frac {1}{\mu-\overline h}\,
[y^{\prime}(\ell)-\overline hy(\ell)]\,[\mu \delta
(x-\ell)+\delta^{\prime}(x-\ell)].
\end{split}
\end{equation}
Moreover, the formulas \eqref{137} establish a one-to-one correspondence between the set of all ($*$)-extensions of the Schr\"odinger operator $T_h$ of the form \eqref{131} and all values $\mu \in \dR\cup\{\infty\}$. One can easily check that the ($*$)-extension $\bA_{\mu,h}$ in \eqref{137} of the non-self-adjoint dissipative Schr\"odinger operator $T_h$, ($\IM h>0$) of the form \eqref{131} satisfies the condition
\begin{equation*}\label{145}
\IM\bA_{\mu, h}=\frac{\bA_{\mu, h} - \bA^*_{\mu, h}}{2i}=(\cdot,g)g,
\end{equation*}
where
\begin{equation}\label{146}
g=\frac{(\IM h)^{\frac{1}{2}}}{|\mu - h|}\,[\mu\delta(x-\ell)+\delta^{\prime}(x-\ell)],
\end{equation}
$\delta(x-\ell)$ and $\delta^{\prime}(x-\ell)$ are the delta-function and
its derivative at the point $\ell$, respectively. Furthermore,
\begin{equation*}\label{147}
(y,g)=\frac{(\IM h)^{\frac{1}{2}}}{|\mu - h|}\ [\mu y(\ell)
-y^{\prime}(\ell)],
\end{equation*}
$y\in \calH_+$, $g\in \calH_-$, where $\calH_+ \subset L_2(\ell,+\infty) \subset \calH_-$ is the triplet of Hilbert spaces introduced above.

Let $y\in\dom(T_h)$, then $y^{\prime}(\ell)=hy(\ell)$ and
$$
\begin{aligned}
\IM\bA_{\mu, h}\, y&=\IM T_h y=(y,g)g=\frac{(\IM h)^{\frac{1}{2}}}{|\mu - h|}\ [\mu y(\ell)-y^{\prime}(\ell)]g\\&=\frac{(\IM h)^{\frac{1}{2}}}{|\mu - h|}\ [\mu y(\ell)-hy(\ell)]g\\
&=\frac{(\IM h)^{\frac{1}{2}}}{|\mu - h|}\ (\mu-h)y(\ell)\cdot\frac{(\IM h)^{\frac{1}{2}}}{|\mu - h|}\,[\mu\delta(x-\ell)+\delta^{\prime}(x-\ell)]\\
&=y(\ell)\frac{(\IM h)(\mu-h)}{|\mu - h|^2}\,[\mu\delta(x-\ell)+\delta^{\prime}(x-\ell)].
\end{aligned}
$$
Consequently,
\begin{equation}\label{e-25}
    \IM (T_h y,y)=y(\ell)\frac{(\IM h)(\mu-h)}{|\mu - h|^2}\cdot (\mu-\bar h)\overline{y(\ell)}=(\IM h)|y(\ell)|^2.
\end{equation}
Having in mind \eqref{e-25} we will call  $\IM h$ a \textbf{coefficient of dissipation} and denote it by $\calD=\IM h$.

It was also shown in \cite{ABT} that the quasi-kernel $\hat A_\xi$ of $\RE\bA_{\mu, h}$ is given by
\begin{equation}\label{e-31}
  \left\{ \begin{array}{l}
 \hat A_\xi y=-y^{\prime\prime}+q(x)y \\
 y^{\prime}(\ell)=\xi y(\ell) \\
 \end{array} \right.,\quad \textrm{where} \quad  \xi=\frac{\mu\RE h-|h|^2}{\mu-\RE h}.
\end{equation}
Let $E=\dC$, $K{c}=cg \;(c\in \dC)$. It is clear that
\begin{equation}\label{148}
K^* y=(y,g),\quad  y\in \calH_+,
\end{equation}
and $\IM\bA_{\mu, h}=KK^*.$ Therefore, the array
\begin{equation}\label{149}
\Theta_{\mu, h}= \begin{pmatrix} \bA_{\mu, h}&K&1\cr \calH_+ \subset
L_2(\ell,+\infty) \subset \calH_-& &\dC\cr \end{pmatrix},
\end{equation}
is an L-system  with the main operator $\bA_{\mu, h}$ of the form \eqref{137} with the channel operator $K$ given by \eqref{148}.
We will say that an L-system $\Theta_{\mu, h}$ of the form \eqref{149} has the  coefficient of dissipation $\calD$ if its main operator $T_h$ has the  coefficient of dissipation $\calD=\IM h$.
It was shown in
 \cite{ABT}, \cite{ArTs0} that the transfer and impedance functions of $\Theta_{\mu, h}$ can be evaluated as
\begin{equation}\label{150}
W_{\Theta_{\mu, h}}(z)= \frac{\mu -h}{\mu - \overline h}\,\,
\frac{m_\infty(z)+ \overline h}{m_\infty(z)+h},
\end{equation}
and
\begin{equation}\label{1501}
V_{\Theta_{\mu, h}}(z)=\frac{\left(m_\infty(z)+\mu\right)\IM h}{\left(\mu-\RE h\right)m_\infty(z)+\mu\RE h-|h|^2}.
\end{equation}

Suppose that the main operator $T_h$ of the L-system \eqref{149} has the von Neumann representation \eqref{parpar} with the parameter $\kappa$ related to some normalized deficiency basis $g_\pm$. 
It was shown in \cite{BMkT} that if the point $z=i$ belongs to the resolvent set $\rho(T_h)$, then in this case (by \eqref{150})
\begin{equation}\label{e-37}
|\kappa|=\frac{1}{|W_{\Theta_{\mu, h}}(i)|}=\left|\frac{\mu - \overline h}{\mu -h}\cdot\frac{m_\infty(i)+h}{m_\infty(i)+ \overline h}\right|=\left|\frac{m_\infty(i)+h}{m_\infty(i)+ \overline h}\right|.
\end{equation}
Taking into account that $\IM h>0$ and $\IM m_\infty(i)<0$ one can easily see that $|\kappa|<1$ in \eqref{e-37}.
\begin{remark}\label{r-7}
Now suppose $h\ne-m_\infty(i)$ and hence $\kappa\ne0$. If $\Theta_{\mu, h}$ satisfies  Hypothesis \ref{setup}, then
\begin{equation}\label{e-38-new}
\kappa=\frac{1}{W_{\Theta_{\mu, h}}(i)}=\frac{\mu - \overline h}{\mu -h}\cdot\frac{m_\infty(i)+h}{m_\infty(i)+ \overline h}
\end{equation}
is real.

We can reverse the logic and try to find a value of parameter $\mu$ for which $\kappa$ in \eqref{e-38-new} is real and thus $\kappa=|\kappa|$. This can be achieved if we set
\begin{equation}\label{e-40}
    \frac{\mu - \overline h}{\mu -h}=e^{-i\alpha},
\end{equation}
where $\alpha\in(-\pi,\pi]$ is such that
$$
\frac{m_\infty(i)+h}{m_\infty(i)+ \overline h}=\left|\frac{m_\infty(i)+h}{m_\infty(i)+ \overline h}\right|e^{i\alpha}.
$$
Clearly, such a choice of the angle $\alpha$ will make $\kappa$ real and $0<\kappa<1$. That is,
\begin{equation}\label{e-41}
    \kappa=\left|\frac{m_\infty(i)+h}{m_\infty(i)+ \overline h}\right|.
\end{equation}
Solving \eqref{e-41} for $\mu$ yields
\begin{equation}\label{e-42}
    \mu=\mu_1=\frac{e^{i\alpha}\bar h-h}{e^{i\alpha}-1},\quad{\textrm{where}}\quad \alpha=\textrm{Arg}\left(\frac{m_\infty(i)+h}{m_\infty(i)+ \overline h} \right).
\end{equation}
Therefore, we can describe an L-system
\begin{equation}\label{e-43}
\Theta_{\mu_1,h}= \begin{pmatrix} \bA_{\mu_1,h}&K_1&1\cr \calH_+ \subset
L_2(\ell,+\infty) \subset \calH_-& &\dC\cr \end{pmatrix},
\end{equation}
with main Schr\"odinger operator $T_h$ that satisfies Hypothesis \ref{setup} with $0<\kappa<1$. The state-space operator $\bA_{\mu_1,h}$ of this L-system will be uniquely  defined via \eqref{137} by the parameter $h$ of $T_h$ and parameter $\mu_1$ given by \eqref{e-42}.

Similarly we can give the description of an L-system
\begin{equation}\label{e-43-1}
\Theta_{\mu_2,h}= \begin{pmatrix} \bA_{\mu_2,h}&K_2&1\cr \calH_+ \subset
L_2(\ell,+\infty) \subset \calH_-& &\dC\cr \end{pmatrix},
\end{equation}
with the same main operator $T_h$ that satisfies the conditions of Hypothesis \ref{setup-1} with $0<\kappa<1$. Applying \eqref{e-55-1} yields
\begin{equation}\label{e-44}
    \mu=\mu_2=\frac{e^{i\alpha}\bar h+h}{e^{i\alpha}+1},\quad{\textrm{where}}\quad \alpha=\textrm{Arg}\left(\frac{m_\infty(i)+h}{m_\infty(i)+ \overline h} \right).
\end{equation}
The state-space operator $\bA_{\mu_2,h}$ from Hypothesis \ref{setup-1} satisfying L-system will be uniquely  defined via \eqref{137} by the parameter $h$ of $T_h$ and parameter $\mu_2$ given by \eqref{e-44}. As  shown in \cite{BMkT-2},
$$V_{\Theta_{\mu_1,h}}(z)\in\sM_\kappa\quad\textit{ and }\quad V_{\Theta_{\mu_2,h}}(z)\in\sM_\kappa^{-1}.$$
\end{remark}

We will rely on formula \eqref{e-37} to prove the following result.
\begin{theorem}\label{t-6}
Let $\Theta_{\mu, h}$ be a Schr\"odinger L-system \eqref{149} with the main operator $T_h$ of the form \eqref{131}. Then for any $\mu\in\dR\cup\{\infty\}$ the impedance function $V_{\Theta_{\mu, h}}(z)$ belongs to the class $\sM$ if and only if $h=-m_\infty(i)$.
\end{theorem}
\begin{proof}
It is well known (see \cite{BMkT}, \cite{BMkT-2}) that the impedance function $V_{\Theta_{\mu, h}}(z)$ belongs to the class $\sM$ if and only if it has the integral representation
\begin{equation}\label{e-31-new}
V_{\Theta_{\mu, h}}(z)=\int_\bbR \left
(\frac{1}{\lambda-z}-\frac{\lambda}{1+\lambda^2}\right)d\sigma(\lambda),\quad \int_\bbR\frac{d\sigma(\lambda)}{1+\lambda^2}=1.
\end{equation}
Suppose that the  Schr\"odinger L-system $\Theta_{\mu, h}$ contains the main operator $T_h$ such that $h=-m_\infty(i)$ and arbitrary $\mu\in\dR\cup\{\infty\}$. Then \eqref{1501} yields
$$
\begin{aligned}
    V_{\Theta_{\mu, h}}&(i)=\frac{-\left(m_\infty(z)+\mu\right)\IM m_\infty(i)}{\left(\mu+\RE m_\infty(i)\right)m_\infty(i)-\mu\RE m_\infty(i)-|m_\infty(i)|^2}\\
    &=\frac{-\left(m_\infty(z)+\mu\right)\IM m_\infty(i)}{\mu\left(\RE m_\infty(i)+i\IM m_\infty(i)\right)+\RE m_\infty(i)\cdot  m_\infty(i)-\mu\RE m_\infty(i)-|m_\infty(i)|^2}\\
    &=\frac{-\left(m_\infty(z)+\mu\right)\IM m_\infty(i)}{i\mu\IM m_\infty(i)+\RE m_\infty(i)\cdot  m_\infty(i)-|m_\infty(i)|^2}\\
    &=\frac{-\left(m_\infty(z)+\mu\right)\IM m_\infty(i)}{i\mu\IM m_\infty(i)+\left(\RE m_\infty(i)-\overline{m_\infty(i)}\right)m_\infty(i)}\\
    &=\frac{-\left(m_\infty(z)+\mu\right)\IM m_\infty(i)}{i\mu\IM m_\infty(i)+\left(\RE m_\infty(i)-\RE m_\infty(i)+i\IM m_\infty(i)\right)m_\infty(i)}\\
    &=\frac{-\left(m_\infty(z)+\mu\right)\IM m_\infty(i)}{i\mu\IM m_\infty(i)\left(\mu+ m_\infty(i)\right)}=-\frac{1}{i}=i.
\end{aligned}
$$
Therefore, $V_{\Theta_{\mu, h}}(z)$ admits  the integral representation \eqref{e-31-new} and hence belongs to the class $\sM$.

Conversely, let  $V_{\Theta_{\mu, h}}(z)$ belong to the class $\sM$ and hence has the integral representation \eqref{e-31-new}. Then $V_{\Theta_{\mu, h}}(z)\in\sM$ and consequently (see \cite{BMkT}) $\kappa=0$, where $\kappa$ is the von Neumann parameter of $T_h$. In this case \eqref{e-37} implies that  $h=-m_\infty(i)$.
\end{proof}
Clearly, it follows from \eqref{e-37} that  $\kappa=0$ if and only if $h=-m_\infty(i)$. Therefore,
\begin{equation}\label{e-38}
\left\{ \begin{array}{l}
 T_h y=-y^{\prime\prime}+q(x)y \\
 y^{\prime}(\ell)=h y(\ell) \\
 \end{array} \right.\qquad\textrm{has\;  $\kappa=0$ if and only if }\; h=-m_\infty(i).
\end{equation}
The state-space operator $\bA_{\mu,h}$ of the L-system \eqref{149} (and its adjoint $\bA^*_{\mu,h}$) in this case will depend on the parameter $\mu$ only and take the form
\begin{equation}\label{137-1}
\begin{split}
&\bA_{\mu, -m_\infty(i)}\, y=-y^{\prime\prime}+q(x)y-\frac {y^{\prime}(\ell)+m_\infty(i)y(\ell)}{\mu+m_\infty(i)}[\mu \delta (x-\ell)+\delta^{\prime}(x-\ell)], \\
&\bA^*_{\mu,-m_\infty(i)}\, y=-y^{\prime\prime}+q(x)y-\frac {y^{\prime}(\ell)+m_\infty(-i)y(\ell)}{\mu+m_\infty(-i)}[\mu \delta(x-\ell)+\delta^{\prime}(x-\ell)].
\end{split}
\end{equation}
It can be checked directly that in this case the impedance function $V_{\Theta_{\mu,h}}(\lambda)$ belongs to the class $\sM$ (see also \cite{BMkT}).

The following theorem is  similar to Theorem \ref{t-6} result for the class $\sM_\kappa$.
\begin{theorem}\label{t-7}
Let $0<a<1$ and $\kappa=\frac{1-a}{1+a}$. Also let $\Theta_{\mu, h}$ be a Schr\"odinger L-system \eqref{149} with the main operator $T_h$ of the form \eqref{131} that has the modulus of von Neumann's parameter  $\kappa$. Then the impedance function $V_{\Theta_{\mu, h}}(z)$  belongs to the class $\sM_\kappa$  if and only if either
\begin{equation}\label{e-36-h}
    h=-\RE m_\infty(i)-\left(\frac{i}{a}\right)\IM m_\infty(i),
\end{equation}
or
\begin{equation}\label{e-37-h}
    h=\frac{a^2d^2\mu-d^2(c+\mu)-c(c+\mu)^2}{a^2d^2+(c+\mu)^2}-i\frac{ad^3+ad(c+\mu)^2}{a^2d^2+(c+\mu)^2},
\end{equation}
where $c=\RE m_\infty(i)$ and $d=\IM m_\infty(i)$.
\end{theorem}
\begin{proof}
It is well known (see  \cite{BMkT-2}) that the impedance function $V_{\Theta_{\mu, h}}(z)$ belongs to the class $\sM_\kappa$ for $\kappa=\frac{1-a}{1+a}$,  $0<a<1$  if and only if it has the integral representation
\begin{equation}\label{e-35-new}
V_{\Theta_{\mu, h}}(z)=\int_\bbR \left
(\frac{1}{\lambda-z}-\frac{\lambda}{1+\lambda^2}\right)d\sigma(\lambda),\quad \int_\bbR\frac{d\sigma(\lambda)}{1+\lambda^2}=a,\quad 0<a<1,
\end{equation}

Suppose that the impedance function $V_{\Theta_{\mu, h}}(z)\in\sM_\kappa$ and thus has the integral representation \eqref{e-35-new} with $0<a<1$. 
Then, \eqref{1501} and \eqref{e-35-new} imply
\begin{equation}\label{e-43-V}
V_{\Theta_{\mu, h}}(i)=\frac{\left(m_\infty(i)+\mu\right)\IM h}{\left(\mu-\RE h\right)m_\infty(i)+\mu\RE h-|h|^2}=ai.
\end{equation}
Setting
\begin{equation}\label{e-44-set}
    c=\RE m_\infty(i),\quad d=\IM m_\infty(i),\quad x=\RE h,\quad y=\IM h,
\end{equation}
transforms \eqref{e-43-V} into
\begin{equation}\label{e-45}
\frac{cy+\mu y+dyi}{(\mu-x)c+\mu x -x^2-y^2+(\mu d-x d)i}=ai.
\end{equation}
Cross multiplying \eqref{e-45} and equating the real parts yields
$cy+\mu y=axd-a\mu d,$
or
\begin{equation}\label{e-46-y}
y=\frac{ad(x-\mu)}{c+\mu}.
\end{equation}
Equating the imaginary parts in \eqref{e-45} leads to
$$
dy=ac(\mu-x)+a\mu x-a x^2-a y^2.
$$
Substituting \eqref{e-46-y} into the above equation results in
$$
\frac{d^2(\mu-x)}{c+\mu}+c(\mu-x)+x(\mu-x)-\frac{a^2d^2(\mu-x)^2}{(c+\mu)^2}=0,
$$
or
$$
(\mu-x)\left( \frac{d^2}{c+\mu}+c+x-\frac{a^2d^2(\mu-x)}{(c+\mu)^2} \right)=0.
$$
Exploring the first solution of the above equation when $\mu=x$ and applying \eqref{e-46-y} leads us to $x=-c$ and $y=-d/a$, or
$$
\RE h=-\RE m_\infty(i) \quad\textrm{ and}\quad \IM h=-\left(\frac{i}{a}\right)\IM m_\infty(i),
$$
that confirms \eqref{e-36-h}.

Assuming that $\mu\ne x$ and setting the second factor to be zero yields
$$
a^2d^2(\mu-x)-c(c+\mu)^2-x(c+\mu)^2-d^2(c+\mu)=0,
$$
that we solve for $x$ to obtain
\begin{equation}\label{e-47-x}
    x=\frac{a^2d^2\mu-d^2(c+\mu)-c(c+\mu)^2}{a^2d^2+(c+\mu)^2}.
\end{equation}
Substituting \eqref{e-47-x} into \eqref{e-46-y} results in
\begin{equation}\label{e-48-y}
    y=-\frac{ad^3+ad(c+\mu)^2}{a^2d^2+(c+\mu)^2}.
\end{equation}
Keeping in mind that $x=\RE h$ and $y=\IM h$ we confirm \eqref{e-37-h}.

Conversely, suppose either \eqref{e-36-h} or \eqref{e-37-h} hold. If \eqref{e-36-h} is true, then substituting this value of $h$ and $\mu=-\RE m_\infty(i)$ into the left side of \eqref{e-43-V} we obtain that $V_{\Theta_{\mu, h}}(i)=ai$. In case if \eqref{e-37-h} is true we also substitute the value of $h$ in  \eqref{e-43-V} and get that $V_{\Theta_{\mu, h}}(i)=ai$ for any real $\mu\ne-\RE m_\infty(i)$. Combining this normalization condition of $V_{\Theta_{\mu, h}}(z)$ with its known integral representation we conclude that $V_{\Theta_{\mu, h}}(z)\in\sM_\kappa$ and is described by \eqref{e-35-new}.
\end{proof}

A very similar result takes place for the class $\sM_\kappa^{-1}$.
\begin{theorem}\label{t-7-1}
Let $a>1$ and $\kappa=\frac{a-1}{1+a}$. Also, let  $\Theta_{\mu, h}$ be a Schr\"odinger L-system \eqref{149} with the main operator $T_h$ of the form \eqref{131} that has the modulus of von Neumann's parameter equal $\kappa$. Then the impedance function $V_{\Theta_{\mu, h}}(z)$
 belongs to the class $\sM_\kappa^{-1}$ with $\kappa=\frac{a-1}{1+a}$ if and only if either \eqref{e-36-h} or \eqref{e-37-h} hold true for $a>1$.
\end{theorem}
\begin{proof}
It is well known (see  \cite{BMkT-2}) that the impedance function $V_{\Theta_{\mu, h}}(z)$ belongs to the class $\sM_\kappa^{-1}$ for $\kappa=\frac{a-1}{1+a}$,  $a>1$  if and only if it
has the integral representation
\begin{equation}\label{e-49-new}
V_{\Theta_{\mu, h}}(z)=\int_\bbR \left
(\frac{1}{\lambda-z}-\frac{\lambda}{1+\lambda^2}\right)d\sigma(\lambda),\quad \int_\bbR\frac{d\sigma(\lambda)}{1+\lambda^2}=a,\quad 1<a.
\end{equation}
The rest of the proof of Theorem \ref{t-7-1} resembles the one of Theorem \ref{t-7} taking into account that $a>1$ this time. Suppose that the impedance function $V_{\Theta_{\mu, h}}(z)\in\sM_\kappa^{-1}$ with $\kappa$ given by \eqref{e-45-kappa-2} and hence has the integral representation \eqref{e-49-new} with $a>1$. Then, as we have shown in the proof of Theorem \ref{t-7}, relations \eqref{1501} and \eqref{e-35-new} imply \eqref{e-43-V} for $a>1$. Following the steps in the proof of Theorem \ref{t-7} we use \eqref{e-43-V} to obtain \eqref{e-47-x} and \eqref{e-48-y} for $a>1$ thus confirming either \eqref{e-36-h} or \eqref{e-37-h}.

Conversely, suppose either \eqref{e-36-h} or \eqref{e-37-h} hold true for $a>1$. Following the steps of the proof of the necessity in Theorem \ref{t-7} we obtain that $V_{\Theta_{\mu, h}}(i)=ai$ for $a>1$ and conclude that $V_{\Theta_{\mu, h}}(z)\in\sM_\kappa^{-1}$.
\end{proof}

\section{On von Neumann's parameter of extremal Schr\"odinger operator $T_h$}\label{s5}

Recall that, the Philips-Kato extension problem is about giving a description of all maximal accretive and/or sectorial extensions $A$ of a non-negative symmetric operator $\dA$ such that $\dA\subset A\subset\dA^*$. In this and the next sections we discuss the Phillips-Kato extension problem for a non-negative Schr\"odinger operator  with deficiency indices $(1, 1)$ in $\calH=L^2(\ell, \infty)$. For the one-dimensional Schr\"odinger operator $T_h$ ($\IM h>0$) on the semi-axis the  Phillips-Kato extension problem in restricted sense was solved by one of the authors in \cite{T87} (see also \cite{AKT}, \cite{AT2009}, \cite{Ts2}). The solution  in terms of boundary values  is presented in Theorem \ref{t-8} below. Our goal, however, is in presenting the necessary condition to the existence of solution for the  Phillips-Kato problem in terms of the modulus of von Neumann's parameter.

Suppose that the symmetric operator $\dA$ of the form \eqref{128} with deficiency indices (1,1) is nonnegative, i.e., $(\dA f,f) \geq 0$ for all $f \in \dom(\dA)$. 
\begin{theorem}[\cite{Ts2}, \cite{T87}, see also \cite{AT2009}]\label{t-8}%
Let $\dA$ be a nonnegative symmetric Schr\"odinger operator of the form \eqref{128} with deficiency indices $(1, 1)$ and locally summable potential in $\calH=L^2(\ell, \infty).$ Consider operator $T_h$ of the form \eqref{131}.  Then
 \begin{enumerate}
\item operator $\dA$ has more than one non-negative self-adjoint extension, i.e., the Friedrichs extension $A_F$ and the Kre\u{\i}n-von Neumann extension $A_K$ do not coincide if and only if $m_{\infty}(-0)<\infty$;
 \item operator $T_h$ coincides with the Kre\u{\i}n-von Neumann extension if and  only if $h=-m_{\infty}(-0)$;
\item operator $T_h$ is accretive if and only if
\begin{equation}\label{138}
\RE h\geq-m_\infty(-0);
\end{equation}
\item operator $T_h$, ($h\ne\bar h$) is $\alpha$-sectorial if and only if  $\RE h >-m_{\infty}(-0)$ holds;
\item operator $T_h$, ($h\ne\bar h$) is accretive extremal if and only if $\RE h=-m_{\infty}(-0)$;
\item if $T_h, (\IM h>0)$ is $\alpha$-sectorial, then the exact angle of sectoriality $\alpha$ can be calculated as
\begin{equation}\label{e10-45}
\tan\alpha=\frac{\IM h}{\RE h+m_{\infty}(-0)}.
\end{equation}

\end{enumerate}
\end{theorem}
It follows from item (2) of Theorem \ref{t-8} that if $m_{\infty}(-0)=\infty$, then the operator $T_h$ turns into a self-adjoint operator $T_\infty$ corresponding to the Dirichlet problem
$$
\left\{ \begin{array}{l}
 T_\infty y=-y^{\prime\prime}+q(x)y \\
 y(\ell)=0. \\
 \end{array} \right.
$$
For the remainder of this paper we assume that $m_{\infty}(-0)<\infty$. Then according to Theorem \ref{t-8} above (see also \cite{Ts81}, \cite{Ts80}) we have the existence of the operator $T_h$, ($\IM h>0$) which is accretive and/or sectorial.
The following  was shown in \cite{ABT}. Let  $T_h \;(\IM h>0)$ be an accretive Schr\"odinger operator of the
form \eqref{131}. Then for all real $\mu$ satisfying the  inequality
\begin{equation}\label{151}
\mu \geq \frac {(\IM h)^2}{m_\infty(-0)+\RE h}+\RE h,
\end{equation}
the operators \eqref{137} form the set of all accretive $(*)$-extensions $\bA$ of the operator $T_h$.
The accretive operator $T_h$ has a unique accretive
$(*)$-extension $\bA$ if and only if
$$\RE h=-m_\infty(-0).$$
In this case this unique $(*)$-extension (and its adjoint) has the form
\begin{equation}\label{153}
\begin{aligned}
&\bA y=-y^{\prime\prime}+q(x)y+[hy(\ell)-y^{\prime}(\ell)]\,\delta(x-\ell), \\
&\bA^* y=-y^{\prime\prime}+q(x)y+[\overline h
y(\ell)-y^{\prime}(\ell)]\,\delta(x-\ell).
\end{aligned}
\end{equation}

Now suppose that $\Theta_{\mu, h}$ of the form \eqref{149} is an L-system with the main Schr\"odinger  operator $T_h$ defined by \eqref{131}.
We are going to tackle the Phillips-Kato extension problem with the help of the von Neumann parameter $\kappa$ of  $T_h$. Assume also that $\mu$ is given by \eqref{e-42} and hence $T_h$  satisfies the  Hypothesis \ref{setup} with real  $\kappa$ given by \eqref{e-41}. Set
\begin{equation}\label{e-43-new}
    m_\infty(i)=A-iB,\; B>0, \; m_\infty(-0)=m,\; C=(A-m)^2,\; D=C+B^2>0.
\end{equation}
Then \eqref{e-41} yields
$$
\begin{aligned}
\kappa&=\left|\frac{m_\infty(i)+h}{m_\infty(i)+ \overline h}\right|=\left|\frac{A-iB+\RE h+i\IM h}{A-iB+\RE h-i\IM h}\right|=\left|\frac{(A+\RE h)+i(\IM h-B)}{(A+\RE h)-i(\IM h+B)}\right|\\
&=\sqrt{\frac{(A+\RE h)^2+(\IM h-B)^2}{(A+\RE h)^2+(\IM h +B)^2}},
\end{aligned}
$$
or
\begin{equation}\label{e-44-new}
    \kappa^2=\frac{(A+\RE h)^2+(\IM h-B)^2}{(A+\RE h)^2+(\IM h +B)^2}.
\end{equation}

Suppose that $T_h$ is an extremal accretive operator. According to Theorem \ref{t-8} we have that $\RE h=-m$. Applying this to \eqref{e-44-new} and using notations \eqref{e-43-new} gives
$$\begin{aligned}
\kappa^2&=\frac{(A-m)^2+(\IM h-B)^2}{(A-m)^2+(\IM h +B)^2}=\frac{C+(\IM h)^2-2B\IM h+B^2}{C+(\IM h)^2+2B\IM h+B^2}\\
&=\frac{(\IM h)^2-2B\IM h+D}{(\IM h)^2+2B\IM h+D}.
\end{aligned}
$$
Consider $\kappa^2$ as a function $f$ of $H=\IM h$, that is
\begin{equation}\label{e-45-new}
f(H)=\kappa^2(H)=\frac{H^2-2B H+D}{H^2+2B H+D}\,.
\end{equation}
Taking the derivative of $f(H)$ in \eqref{e-45-new} and simplifying yields
$$
f'(H)=\frac{4B(H^2-D)}{(H^2+BH+D)^2}.
$$
Setting $f'(H)=0$ and keeping in mind that $H=\IM h>0$ and $D>0$ we obtain a critical number
$$
H=\sqrt D,
$$
that can be checked to be a point of minimum of $f(H)$ for all $H>0$. Also, direct substitution gives
$$
f(\sqrt D)=\frac{\sqrt D-B}{\sqrt D+B}.
$$

Consequently,
$$
\frac{\sqrt D-B}{\sqrt D+B}\le\kappa^2<1,
$$
and hence
$$
    \sqrt{\frac{\sqrt D-B}{\sqrt D+B}}\le\kappa<1,
$$
or, after backward substitution and simplification,
\begin{equation}\label{e-46}
    \sqrt{\frac{\sqrt{|m_\infty(i)|^2-2m_\infty(-0)\RE\,m_\infty(i)+m_\infty^2(-0)} +\Im\, m_\infty(i)}{\sqrt{|m_\infty(i)|^2-2m_\infty(-0)\RE\,m_\infty(i)+m_\infty^2(-0)}-\Im\, m_\infty(i)}}\le\kappa<1.
\end{equation}
Thus, $T_h$ is extremal accretive operator if and only if \eqref{e-46} holds, that is, the von Neumann parameter $\kappa$ of $T_h$ is such that $\kappa_0\le\kappa<1$, with
\begin{equation}\label{e-47}
    \kappa_0=\sqrt{\frac{\sqrt{|m_\infty(i)|^2-2m_\infty(-0)\RE\,m_\infty(i)+m_\infty^2(-0)} +\Im\, m_\infty(i)}{\sqrt{|m_\infty(i)|^2-2m_\infty(-0)\RE\,m_\infty(i)+m_\infty^2(-0)}-\Im\, m_\infty(i)}}.
\end{equation}
A sample graph of $\kappa$ as a function of $\IM h$ is shown in Figure \ref{fig-1}.

We can summarize the above reasoning in the following theorem.
\begin{theorem}\label{t-9-1}
Under the condition \textrm{(1)} of Theorem \ref{t-8}, if a Schr\"odinger operator $T_h$ of the form \eqref{131} with the modulus of von Neumann's parameter $\kappa$ is extremal accretive, then $\kappa_0\le\kappa<1$, where $\kappa_0$ is given by \eqref{e-47}.
\end{theorem}
The proof follows from part (5) of Theorem \ref{t-8} and formulas \eqref{e-46} and \eqref{e-47}.

The following lemma  contains a  useful property of the function $m_\infty(z)$ that makes  $\kappa_0$  (given by \eqref{e-47}) vanish.

\begin{lemma}\label{l-13}
Under the conditions  of Theorem \ref{t-9-1}, the value of $\kappa_0$ given by \eqref{e-47} is zero if and only if
\begin{equation}\label{e-67-Am}
    \RE m_\infty(i)= m_\infty(-0).
\end{equation}
\end{lemma}
\begin{proof}
Following our notation  that we set in \eqref{e-43-new}, from \eqref{e-47} we get that
\begin{equation}\label{e-68-k0}
    \kappa_0=\sqrt{\frac{\sqrt D-B}{\sqrt D+B}}
\end{equation}
Since $A=\RE m_\infty(i)$ and $m= m_\infty(-0)$, we need to show that $\kappa_0=0$ if and only if $A=m$.

If $A=m$, then using \eqref{e-43-new} we obtain
$$
\sqrt D-B=\sqrt{C+B^2}-B=\sqrt{(A-m)^2+B^2}-B=\sqrt{0+B^2}-B=0,
$$
and hence the numerator of the fraction in \eqref{e-68-k0} is zero making $\kappa_0=0$.

Conversely, assume that $\kappa_0=0$. Then $\sqrt D-B=0$ and hence $D=B^2$. This implies $C=A-m=0$ yielding  \eqref{e-67-Am}.
\end{proof}

Now we will state and prove an inverse version of Theorem \ref{t-9-1}.
\begin{theorem}\label{t-9-2}
Let $\dA$ be the same as in  Theorems \ref{t-8} and \ref{t-9-1}. If $\kappa_0\le\kappa<1$, with $\kappa_0$ given by \eqref{e-47}, then there exist two (just one if $\kappa=\kappa_0$)  extremal dissipative Schr\"odinger operators $T_h$, ($\dA \subset T_h \subset \dA^*$), of the form \eqref{131} such that their modulus of von Neumann's parameters equals  $\kappa$.

In this case, the boundary value(s) of $h$  are given by
\begin{equation}\label{e-66-h}
    h=-m_{\infty}(-0)+\frac{B(1+\kappa^2)\pm\sqrt{4{B}^2\kappa^2-C(1-\kappa^2)^2}}{1-\kappa^2}\,i,
\end{equation}
where $B$ and $C$ are given by \eqref{e-43-new}.
 \end{theorem}
\begin{proof}
Let $T_h$ be a Schr\"odinger operator of the form \eqref{131} with the same potential as  in \eqref{128}. All we need is to show that there exists a value of $h$ that makes $T_h$ an extremal dissipative quasi-self-adjoint extension of $\dA$ whose modulus of von Neumann's parameter is equal to the given $\kappa\in[\kappa_0,1)$. Note that equation \eqref{e-44-new} holds for any $T_h$ with von Neumann's parameter  $\kappa$. Setting $\RE h=-m$ in \eqref{e-44-new} will guarantee (see Theorem \ref{t-8}) that $T_h$ is extremal and yields \eqref{e-45-new}, that is,
\begin{equation}\label{e-14-old}
\kappa^2=\frac{{x}^2-2{B} {x}+{D}}{{x}^2+2{B} {x}+{D}},
\end{equation}
where ${x}=\IM h$. Modifying \eqref{e-14-old} leads to the quadratic equation
\begin{equation}\label{e-14}
x^2-2{B}\left(\frac{1+\kappa^2}{1-\kappa^2}\right)x+{D}=0.
\end{equation}
We are going to show that equation \eqref{e-14} has at least one real positive solution. Setting
$$\xi=\frac{1+\kappa^2}{1-\kappa^2}\,,$$
note that when $\kappa\in[\kappa_0,1)$ we have $\xi\in[\xi_0,+\infty)$, where
$$
\xi_0=\frac{1+\frac{\sqrt {D}-{B}}{\sqrt {D}+{B}}}{1-\frac{\sqrt {D}-{B}}{\sqrt {D}+{B}}}=\frac{\sqrt{D}}{B}.
$$
Consider the discriminant of the quadratic equation \eqref{e-14} as a function of $\xi$
$$
f(\xi)=4{B}^2\xi^2-4{D}.
$$
Then its derivative $f'(\xi)=8{B^2}\xi$ is always positive on $\xi\in[\xi_0,+\infty)$ indicating that $f(\xi)$ is an increasing function of $\xi\in[\xi_0,+\infty)$. Moreover, the direct check reveals that
$$f(\xi_0)=4{B}^2\xi_0^2-4{D}=0,$$
and hence $f(\xi)$ takes positive values at $\xi\in(\xi_0,+\infty)$. Applying the quadratic formula to equation \eqref{e-14} and taking into account that ${B}>0$, ${D}>0$ and $\xi>0$, we obtain that
\begin{equation}\label{e-15}
x=\frac{2{B}\xi\pm\sqrt{4{B}^2\xi^2-4{D}}}{2}={B}\xi\pm\sqrt{{B}^2\xi^2-{D}}
\end{equation}
yields two positive real solutions. Therefore, the value of $h$ such that $\RE h=-m_{\infty}(-0)$ and $\IM h$ equal to one of the values of $x$ from \eqref{e-15} is the one that makes $T_h$ an extremal dissipative quasi-self-adjoint extension of $\dA$. Substituting the expression for $\xi$ into \eqref{e-15} and simplifying yields \eqref{e-66-h}.

Clearly, if $\kappa\in(\kappa_0,1)$, then  \eqref{e-66-h} yields two distinct values of the boundary parameter $h$  and hence there are two different Schr\"odinger operators $T_h$ with von Neumann's parameter $\kappa$. In the case when $\kappa=\kappa_0$ the expression under the radical in \eqref{e-66-h} turns into zero and hence  there is only one value of $h$ yielding a unique extremal Schr\"odinger operator $T_h$ with von Neumann's parameter $\kappa=\kappa_0$.
\end{proof}
\begin{figure}[ht]
  \begin{center}
  \includegraphics[width=110mm]{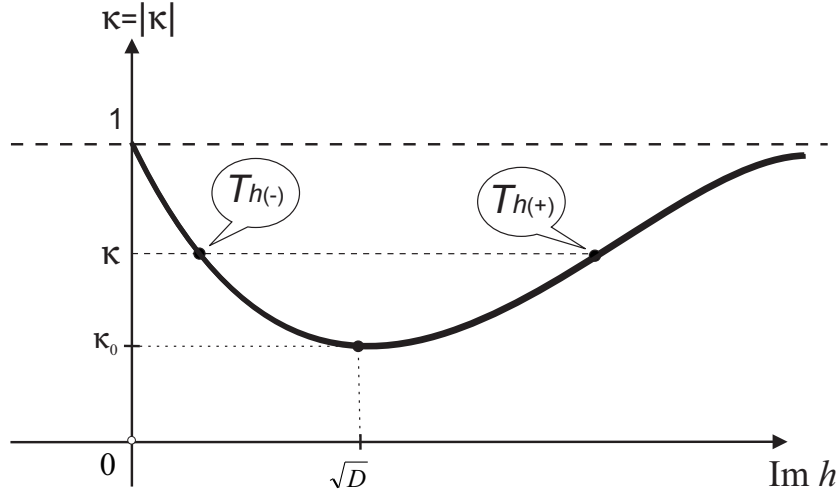}
  \caption{Two extremal operators $T_h$ with the same $\kappa$}\label{fig-1}
  \end{center}
\end{figure}

Figure \ref{fig-1} gives good visual illustration of Theorem  \ref{t-9-2}. There are two extremal operators $T_{h(-)}$ and $T_{h(+)}$ shown that share the same (modulus of) von Neumann's parameter $\kappa$. Here $h(-)$ and $h(+)$ are the values of the parameter $h$ in \eqref{e-66-h} that correspond to the $(-)$ and $(+)$ signs of the formula.

\section{On von Neumann's parameter of sectorial Schr\"odinger operator $T_h$}\label{s6}

The main goal of this section is (for a given $\beta\in(0,\pi/2)$) to explicitly describe the interval for the modulus of von Neumann's parameter $\kappa$ of the  operator $T_h$ such that $T_h$ is $\beta$-sectorial.   Throughout this section we keep the assumption of the previous section that the symmetric operator $\dA$ of the form \eqref{128} with deficiency indices (1,1) is nonnegative.

Consider $T_h$ of the form \eqref{131} and assume that $T_h$ is $\beta$-sectorial. Then according to Theorem \ref{t-8} we have that $\RE h >-m_{\infty}(-0)$ and (see \eqref{e10-45})
$$
\cot\beta=\frac{\RE h+m_{\infty}(-0)}{\IM h},
$$
or equivalently
\begin{equation}\label{e-48}
    \RE h+m_{\infty}(-0)=\cot\beta\cdot\IM h.
\end{equation}
As  in Section \ref{s5} we assume that $\Theta$  is an L-system of the form \eqref{149} with the main $\beta$-sectorial Schr\"odinger  operator $T_h$. Suppose also that $\mu$ is given by \eqref{e-42} and hence $T_h$  satisfies  Hypothesis \ref{setup} with real  $\kappa$ given by \eqref{e-41}. Once again, to simplify the calculation process we use our conventions described in \eqref{e-43-new}. Combining \eqref{e-44-new} with \eqref{e-48} yields
$$\begin{aligned}
&\kappa^2=\frac{(A+\RE h)^2+(\IM h-B)^2}{(A+\RE h)^2+(\IM h +B)^2}=\frac{(A-m+\cot\beta\cdot\IM h)^2+(\IM h -B)^2}{(A-m+\cot\beta\cdot\IM h)^2+(\IM h +B)^2}\\
&=\frac{(A-m)^2+2(A-m)\cot\beta\cdot\IM h+\cot^2\beta(\IM h)^2+(\IM h)^2-2B\IM h+B^2}{(A-m)^2+2(A-m)\cot\beta\cdot\IM h+\cot^2\beta(\IM h)^2+(\IM h)^2+2B\IM h+B^2}\\
&=\frac{(\cot^2\beta+1)(\IM h)^2+2((A-m)\cot\beta-B)\IM h+(A-m)^2+B^2}{(\cot^2\beta+1)(\IM h)^2+2((A-m)\cot\beta+B)\IM h+(A-m)^2+B^2}\\
&=\frac{(\IM h)^2+2\sin^2\beta((A-m)\cot\beta-B)\IM h+D\sin^2\beta}{(\IM h)^2+2\sin^2\beta((A-m)\cot\beta+B)\IM h+D\sin^2\beta}.
\end{aligned}
$$
As in Section \ref{s5} we consider $\kappa^2$ as a function $f$ of $H=\IM h$, that is
\begin{equation}\label{e-49}
f(H)=\kappa^2(H)=\frac{H^2+2\sin^2\beta((A-m)\cot\beta-B)H+D\sin^2\beta}{H^2+2\sin^2\beta((A-m)\cot\beta+B)H+D\sin^2\beta}.
\end{equation}
Taking the derivative of $f(H)$ in \eqref{e-49} we get
$$
f'(H)=\frac{4\sin\beta\cdot B(H^2-D\sin^2\beta)}{(H^2+2\sin^2\beta((A-m)\cot\beta+B)H+D\sin^2\beta)^2}.
$$
Setting $f'(H)=0$ and keeping in mind that $H=\IM h>0$ and $D>0$ we obtain a critical number
$$
H=\sin\beta\cdot\sqrt D,
$$
that can be checked to be a point of minimum of $f(H)$ for all $H>0$. Also, direct substitution gives
$$
f(\sin\beta\sqrt D)=\frac{\sqrt D+(A-m)\cos\beta-B\sin\beta}{\sqrt D+(A-m)\cos\beta+B\sin\beta}.
$$

Consequently,
$$
\frac{\sqrt D+(A-m)\cos\beta-B\sin\beta}{\sqrt D+(A-m)\cos\beta+B\sin\beta}\le\kappa^2<1,
$$
and hence
$$
    \sqrt{\frac{\sqrt D+(A-m)\cos\beta-B\sin\beta}{\sqrt D+(A-m)\cos\beta+B\sin\beta}}\le\kappa<1,
$$
or, after backward substitution and simplification, $\kappa_0\le\kappa<1$, where
\begin{equation}\label{e-50}
    \kappa_0=\sqrt{\frac{\sqrt D+(A-m)\cos\beta-B\sin\beta}{\sqrt D+(A-m)\cos\beta+B\sin\beta}}.
\end{equation}
In the above formula
\begin{equation}\label{e-51}
    \begin{aligned}
\sqrt D&=\sqrt{|m_\infty(i)|^2-2m_\infty(-0)\RE\,m_\infty(i)+m_\infty^2(-0)},\\
A-m&=\RE\,m_\infty(i)-m_\infty(-0),\qquad B=-\IM\,m_\infty(i).\\
    \end{aligned}
  \end{equation}
It is easy to see that when $\beta=\pi/2$, then  \eqref{e-50} and \eqref{e-51} match the previous (extremal) case in the formula \eqref{e-47}.
A sample graph of $\kappa$ as a function of $\IM h$ is shown in Figure \ref{fig-2}.

The following two theorems contain the main result of this section.
\begin{theorem}\label{t-10-1}
Under the condition \textrm{(1)} of Theorem \ref{t-8}, if a Schr\"odinger operator $T_h$ of the form \eqref{131} with the modulus of  von Neumann's parameter $\kappa$ is $\beta$-sectorial ($\beta\in(0,\pi/2)$), then $0<\kappa_0\le\kappa<1$, where $\kappa_0$ is given by \eqref{e-50} and \eqref{e-51}.
\end{theorem}
\begin{proof}
The proof follows from part (4) of Theorem \ref{t-8} and the reasoning of this section. We just note that $0<\kappa_0$ because of the fact that a $\beta$-sectorial ($\beta\in(0,\pi/2)$) operator cannot have the von Neumann parameter equal to zero (see Lemma \ref{l-2}).
\end{proof}

Now we will state and prove an inverse version of Theorem \ref{t-10-1} that is the analog of Theorem \ref{t-9-2} from Section \ref{s5}.
\begin{theorem}\label{t-15}
Let $\dA$ be the same as in Theorem \ref{t-10-1}. Given $\beta\in(0,\pi/2)$ and $0<\kappa_0\le\kappa<1$, where $\kappa_0$ is given by \eqref{e-50}, \eqref{e-51}, there exist two (one if $\kappa=\kappa_0$)  $\beta$-sectorial dissipative Schr\"odinger operators $T_h$, ($\dA \subset T_h \subset \dA^*$), of the form \eqref{131} such that the modulus of their von Neumann's parameters equals  $\kappa$. Moreover, the boundary value(s) of $h$ are given by
\begin{equation}\label{e-74-h}
    h=-m+\sin^2\beta(\cot\beta+i)\left(\xi B-(A-m)\cot\beta\pm\sqrt{E}\right),
\end{equation}
where  $m$, $B$ and $C$ are given by \eqref{e-43-new},  $\xi=\frac{1+\kappa^2}{1-\kappa^2}$ and
\begin{equation}\label{e-75-E}
E=\left[(A-m)\cot\beta-\xi B\right]^2-{D}\csc^2\beta.
\end{equation}
 \end{theorem}
\begin{proof}
We are going to use a similar to the proof of Theorem \ref{t-9-2} method to confirm \eqref{e-74-h}. We notice that the rational function in formula \eqref{e-49} resembles the one in \eqref{e-14-old}. Therefore, the following formal change of coefficients
\begin{equation}\label{e-74-b}
B_1=\sin^2\beta((A-m)\cot\beta-B),\quad B_2=\sin^2\beta((A-m)\cot\beta+B),
\end{equation}
leads to
\begin{equation}\label{e-75-new}
\kappa^2=\frac{{x}^2+2{B_1} {x}+{D\sin^2\beta}}{{x}^2+2{B_2} {x}+{D\sin^2\beta}},
\end{equation}
where ${x}=\IM h$. Modifying \eqref{e-75-new} brings us to the quadratic equation
$$
(\kappa^2-1)x^2+2(\kappa^2 B_2-B_1)x+(\kappa^2-1)D\sin^2\beta=0.
$$
Taking into account that $\kappa^2 B_2-B_1=\sin^2\beta[(\kappa^2-1)(A-m)\cot\beta+(\kappa^2+1)B]$ we obtain
$$
x^2+2\sin^2\beta\left[(A-m)\cot\beta+\frac{\kappa^2+1}{\kappa^2-1}B\right]x+D\sin^2\beta=0.
$$
Setting $$\xi=\frac{1+\kappa^2}{1-\kappa^2}\,,$$   note that when $\kappa\in[\kappa_0,1)$ we have $\xi\in[\xi_0,+\infty)$, where (see \eqref{e-50})
$$
\xi_0=\frac{1+\kappa^2_0}{1-\kappa^2_0}=\frac{1+\frac{\sqrt D+(A-m)\cos\beta-B\sin\beta}{\sqrt D+(A-m)\cos\beta+B\sin\beta}}{1-\frac{\sqrt D+(A-m)\cos\beta-B\sin\beta}{\sqrt D+(A-m)\cos\beta+B\sin\beta}}=\frac{\sqrt{D}+(A-m)\cos\beta}{B\sin\beta}.
$$
It is easy to see that if $\beta=\pi/2$, then the value of $\xi_0$ above matches the corresponding value of $\xi_0$ in the proof of Theorem \ref{t-9-2} written for the extremal case. Furthermore,  the above quadratic equation transforms into (see \eqref{e-74-b})
\begin{equation}\label{e-77-new}
(\csc^2\beta)    x^2+2\left[(A-m)\cot\beta-\xi B\right]x+D=0.
\end{equation}
Consider the discriminant of the quadratic equation \eqref{e-77-new} as a function of $\xi$
$$
f(\xi)=4\left[(A-m)\cot\beta-\xi B\right]^2-4{D}\csc^2\beta.
$$
Then its derivative $f'(\xi)=8{B^2}\xi-8B(A-m)\cot\beta$ gives rise to  a critical point $\xi_c=\frac{A-m}{B}\cot\beta$ that is clearly smaller than $\xi_0$, that is $\xi_c<\xi_0$. Consequently, $f'(\xi)$ is always positive on $\xi\in[\xi_0,+\infty)$ indicating that $f(\xi)$ is an increasing function of $\xi\in[\xi_0,+\infty)$. Moreover, the direct check reveals that
$$f(\xi_0)=4{\left[(A-m)\cot\beta-\xi_0 B\right]}^2-4{D}(\csc^2\beta)=4\left[\frac{-\sqrt{D}}{\sin\beta}\right]^2-4D\csc^2\beta=0,$$
and hence $f(\xi)$ takes positive values on $\xi\in(\xi_0,+\infty)$. Applying the quadratic formula to equation \eqref{e-77-new} and taking into account that ${B}>0$, ${D}>0$ and $\xi>0$, we obtain that
\begin{equation}\label{e-78-new}
    \begin{aligned}
x&=\frac{2\left[\xi B-(A-m)\cot\beta\right]\pm\sqrt{4\left[(A-m)\cot\beta-\xi B\right]^2-4{D}\csc^2\beta}}{2\csc^2\beta}\\
&=\sin^2\beta\left(\xi B-(A-m)\cot\beta\pm\sqrt{\left[(A-m)\cot\beta-\xi B\right]^2-{D}\csc^2\beta}\right)\\
&=\sin^2\beta\left(\xi B-(A-m)\cot\beta\pm\sqrt{E}\right)
    \end{aligned}
\end{equation}
yields two positive real solutions. Therefore, the value of $h$ such that $\RE h=-m_{\infty}(-0)+(\cot\beta)x$, where  $x=\IM h$ is given by \eqref{e-78-new},  defines $\beta$-sectorial dissipative Schr\"odinger operators $T_h$. Substituting the expression for $\xi$ into \eqref{e-78-new} and simplifying yields \eqref{e-74-h}.

    It can be clearly seen from  \eqref{e-74-h}  that if $\kappa\in(\kappa_0,1)$, then there are two distinct values of the boundary parameter $h$ and hence two different $\beta$-sectorial  Schr\"odinger operators $T_h$ with von Neumann's parameter. In the case when $\kappa=\kappa_0$ there is only one $\beta$-sectorial  Schr\"odinger operators $T_h$ with this property.
\end{proof}
\begin{figure}[ht]
  \begin{center}
  \includegraphics[width=110mm]{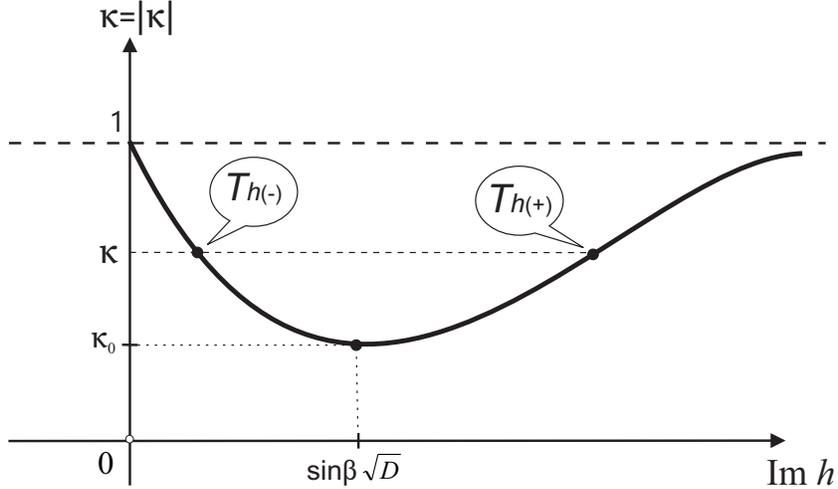}
  \caption{Two $\beta$-sectorial operators $T_h$ with the same $\kappa$}\label{fig-2}
  \end{center}
\end{figure}
Figure \ref{fig-2} gives good visualization  of Theorem  \ref{t-15}. It depicts two $\beta$-sectorial operators $T_{h(-)}$ and $T_{h(+)}$ shown that share the same (modulus of) von Neumann's parameter $\kappa$. Here $h(-)$ and $h(+)$ are the values of the parameter $h$ in \eqref{e-74-h} that correspond to the $(-)$ and $(+)$ signs of the formula.

We finalize the section with a  lemma that contains a very useful property of the function $m_\infty(z)$ related to a nonnegative symmetric operator $\dA$ of the form \eqref{128}.
\begin{lemma}\label{l-16}
Let $\dA$ be a nonnegative symmetric Schr\"odinger operator of the form \eqref{128} with deficiency indices $(1, 1)$ and locally summable potential in $\calH=L^2(\ell, \infty).$ If the corresponding Weyl function $m_\infty(z)$ has the property $m_{\infty}(-0)<\infty$, then
\begin{equation}\label{e-75}
    \RE m_\infty(i)\ge m_\infty(-0).
\end{equation}
\end{lemma}
\begin{proof}
We stick to the notation that we set in \eqref{e-43-new}, that is $m_\infty(i)=A-iB$, $B>0$, $m_\infty(-0)=m$, $C=(A-m)^2$, and $D=C+B^2>0$. We need to show that $A\ge m$. Assume the contrary, $A< m$. Since $m<\infty$, for every $\beta\in(0,\pi/2)$ there exists a $\beta$-sectorial operator $T_h$ of the form \eqref{131} (see \cite{ABT}, \cite{T87}). Take a $\beta_0$ such that
\begin{equation}\label{e-106-beta}
\cos\beta_0=\frac{m-A}{\sqrt{D}} \quad\textrm{ or }\quad \beta_0=\arccos\left( \frac{m-A}{\sqrt{D}}\right).
\end{equation}
Since under our assumption $A< m$, then $\beta_0\in(0,\pi/2)$ and hence there is a $\beta_0$-sectorial operator $T_h$ of the form \eqref{131}. Moreover, according to Theorem \ref{t-15} the boundary value $h$ of this $T_h$ is determined by \eqref{e-74-h} and its von Neumann's parameter is $\kappa_0$ given by  \eqref{e-50}. Using the basic trig identity and \eqref{e-43-new} we get
\begin{equation}\label{e-83-sin}
\sin\beta_0=\sqrt{1-\cos^2\beta_0}=\sqrt{1-\frac{C}{D}}=\frac{\sqrt{D-C}}{\sqrt{D}}=\frac{\sqrt{B^2}}{\sqrt{D}}=\frac{B}{\sqrt{D}}.
\end{equation}
Substituting $\cos\beta_0$ and $\sin\beta_0$ from \eqref{e-106-beta} and \eqref{e-83-sin} into \eqref{e-50} reveals
$$
\kappa_0^2=\frac{\sqrt{D}-\sqrt{C}\cdot\frac{\sqrt{C}}{\sqrt{D}}-B\cdot\frac{{B}}{\sqrt{D}}}{\sqrt{D}-\sqrt{C}\cdot\frac{\sqrt{C}}{\sqrt{D}}+B\cdot\frac{{B}}{\sqrt{D}}}
=\frac{D-C-B^2}{D+B^2-C}=\frac{D-D}{D+B^2-C}=0,
$$
or $\kappa_0=0$. Hence under our assumption we have the existence of a $\beta_0$-sectorial operator $T_h$ with von Neumann's parameter $\kappa_0=0$. This contradicts Lemma \ref{l-2} and thus our assumption that $A< m$ is wrong. Consequently, $A\ge m$.
\end{proof}

\section{c-Entropy of an  L-system and minimal  dissipation coefficient}\label{s7}

We begin with introducing a concept of the c-entropy of an L-system.
\begin{definition}
Let $\Theta$ be an L-system of the form \eqref{e-62}. The quantity
\begin{equation}\label{e-80-entropy-def}
    \calS=-\ln (|W_\Theta(-i)|),
\end{equation}
where $W_\Theta(z)$ is the transfer function of $\Theta$, is called the \textbf{coupling entropy} (or \textbf{c-entropy}) of the L-system $\Theta$.
\end{definition}
There is an alternative operator-theoretic way to define the c-entropy. If $T$ is the main operator of the L-system  $\Theta$ and $\kappa$ is  von Neumann's parameter of $T$ in some basis $g_\pm$, then (see  \cite{BMkT-2}) $|W_\Theta(-i)|=|\kappa|$ and hence
\begin{equation}\label{e-70-entropy}
    \calS=-\ln (|W_\Theta(-i)|)=-\ln(|\kappa|).
\end{equation}

We emphasize that c-entropy defined by formula \eqref{e-70-entropy} does not depend on the choice of deficiency basis $g_\pm$ and moreover is an additive function with respect to the coupling of L-systems (see \cite{BMkT-2}).
 Note that if, in addition,  the point $z=i$ belongs to $\rho(T)$, then we also have that
\begin{equation}\label{e-80-entropy}
     \calS=\ln (|W_\Theta(i)|)=\ln (1/|\kappa|)=-\ln(|\kappa|).
\end{equation}
This follows from the known (see \cite{ABT}) property of the transfer functions for L-systems  that states that $W_\Theta(z)\overline{W_\Theta(\bar z)}=1$ and  the fact that $|W_\Theta(i)|=1/|\kappa|$ (see \cite{BMkT}).

Next we pose for the following optimization  problem associated with the c-entropy of Schr\"o-inger L-systems:
\begin{itemize}
      \item
\textit{Describe L-systems with the Schr\"odinger operator with a given c-entropy and minimal  dissipation coefficient}.
\end{itemize}

The main goal of this section is to provide a solution of the problem above for the L-systems whose impedance functions belong to the (generalized) Donoghue classes $\sM$, $\sM_\kappa$, and $\sM_\kappa^{-1}$.

Let $\Theta_{\mu,h}$ be a Schr\"odinger L-system one of  \eqref{149} based upon the symmetric operator $\dA$ of the form \eqref{128} with corresponding Weyl function  $m_\infty(z)$. Also, let $\kappa$ denote the modulus of the von Neumann parameter of the main operator $T_h$ of the L-system $\Theta_{\mu,h}$. Suppose that the impedance function $V_{\Theta_{\mu,h}}(z)$ given by \eqref{1501},  belongs to a generalized Donoghue class. Then
\begin{equation}\label{e-83-a}
V_{\Theta_{\mu,h}}(i)=ai,
\end{equation}
for some $a>0$. Representation \eqref{e-83-a} has been justified in the proof of Theorem \ref{t-7} for $0<a<1$ (see \eqref{e-43-V}) and can be verified for $a>1$ (cf. Theorem \ref{t-7-1}). Depending on the value of $a$ in \eqref{e-83-a} we (see \eqref{e-45-kappa-1} and \eqref{e-45-kappa-2}) have
\begin{equation}\label{e-57-k}
    \kappa=\left\{
               \begin{array}{ll}
                 \frac{1-a}{1+a}, & \hbox{$0<a<1$;} \\
                 0, & \hbox{$a=1$;} \\
                 \frac{a-1}{1+a}, & \hbox{$a>1$.}
               \end{array}
             \right.
\end{equation}
In order to address the problem above we point out that in accordance with \eqref{e-70-entropy} the knowledge of c-entropy fixes the value of $\kappa=|\kappa|$ and hence of $a$ as well.
 Moreover, by Theorem  \ref{t-7} for the coefficient of dissipation $\IM h$ we have the explicit representation
\begin{equation}\label{e-58-dis}
    \IM h(\mu)=-\frac{ad^3+ad(c+\mu)^2}{a^2d^2+(c+\mu)^2},
\end{equation}
where $c=\RE m_\infty(i)$,  $d=\IM m_\infty(i)$, and $\mu$ is the parameter that determines the L-system $\Theta_{\mu,h}$. Thus, to minimize $\IM h$ for a given $a>0$ we seek the corresponding critical value of $\mu$. Taking the derivative of $\IM h$ as a function of $\mu$, simplifying, and setting it equal to zero yields
$$
\frac{d}{d\mu}\left(\IM h \right)=\frac{2ad^3(c+\mu)(1-a^2)}{(a^2d^2+(c+\mu)^2)^2}=0.
$$
Therefore,
$$\mu=-c=-\RE m_\infty(i)$$
is the only critical point of $\IM h(\mu)$. It is easy to check that this critical value of $\mu$ gives the minimum for $\IM h$ whenever $a>1$ and the maximum whenever $0<a<1$. Also, clearly $\IM h\equiv -d=-\IM m_\infty(i)$ if $a=1$. Moreover, as we have shown in the proof of Theorem \ref{t-7}, in this case
\begin{equation}\label{e-59}
\IM h=-(1/a)\IM m_\infty(i).
\end{equation}
This reasoning  leads  to the following result.
\begin{theorem}\label{t-10}
Assume that  an L-system  $\Theta_{\mu,h}$ of the form  \eqref{149} with a  main operator $T_h$ has c-entropy $\calS$.  Then, under the constraint that the impedance function $V_{\Theta_{\mu,h}}(z)$ belongs to the generalized Donoghue class $\sM_{\kappa}^{-1}$ with $\kappa=e^{-\calS}$, the L-system $\Theta_{\mu,h}$   has the minimum coefficient of dissipation
\begin{equation}\label{e-59-dis}
\calD_{min}=\IM h=-\left(\frac{e^{\calS}-1}{e^{\calS}+1}\right)\IM m_\infty(i)=-\tanh \left(\frac{\calS}{2}\right)\IM m_\infty(i).
\end{equation}
In this case $\Theta_{\mu,h}$ is uniquely determined by the parameters
\begin{equation}\label{e-61-minsys}
    h=-\RE m_\infty(i)-i\tanh \left(\frac{\calS}{2}\right)\IM m_\infty(i)\quad and \quad \mu=-\RE m_\infty(i).
\end{equation}
\end{theorem}
\begin{proof}
 Since out L-system $\Theta_{\mu,h}$ has a fixed given c-entropy $\calS$, then the modulus of von Neumann's parameter $\kappa$ of its main operator $T_h$ is given by $\kappa=e^{-\calS}$. Moreover, since $V_{\Theta_{\mu,h}}(z)\in\sM_\kappa^{-1}$, then $V_{\Theta_{\mu,h}}(i)=ai$ for some $a>1$. Using this value of $\kappa$ and formula \eqref{e-57-k} we obtain
$$
a=\frac{1+\kappa}{1-\kappa}>1.
$$
We have
$$
\frac{1}{a}=\frac{1-e^{-\calS}}{1+e^{-\calS}}=\frac{e^{\calS}-1}{e^{\calS}+1}=\tanh \left(\frac{\calS}{2}\right).
$$
 The rest follows from the fact that $\mu=-c=-\RE m_\infty(i)$ is the point of minimum of $\IM h$ if $a>1$, formulas \eqref{e-36-h} (in the context of Theorem \ref{t-7-1}) and \eqref{e-59}.
\end{proof}
The L-system $\Theta_{\mu,h}$ with the minimum coefficient of dissipation in Theorem \ref{t-10} can be written explicitly. In particular its state space operator $\bA_{\mu,h}$  can be described by  formulas \eqref{137} with the values of $h$ and $\mu$ given by \eqref{e-61-minsys}. That is,
\begin{equation}\label{e-88-new}
\begin{split}
&\bA_{\mu, h}\, y=-y^{\prime\prime}+q(x)y+\frac {y^{\prime}(\ell)-hy(\ell)}{\RE m_\infty(i)+h}\,[\delta^{\prime}(x-\ell)-\RE m_\infty(i) \delta (x-\ell)], \\
&\bA^*_{\mu, h}\, y=-y^{\prime\prime}+q(x)y+\frac {y^{\prime}(\ell)-\overline hy(\ell)}{\RE m_\infty(i)+\overline h}\,[\delta^{\prime}(x-\ell)-\RE m_\infty(i) \delta(x-\ell)],
\end{split}
\end{equation}
where $$h=-\RE m_\infty(i)-i\tanh \left(\frac{\calS}{2}\right)\IM m_\infty(i).$$

An analogues result similar to Theorem \ref{t-10} takes place whenever $0<a<1$.
\begin{theorem}\label{t-11}
Assume that  an L-system  $\Theta_{\mu,h}$ of the form  \eqref{149} with a  main operator $T_h$ has c-entropy $\calS$.  Then, under the constraint  that the impedance function $V_{\Theta_{\mu,h}}(z)$ belongs to the generalized Donoghue class $\sM_{\kappa}$ with $\kappa=e^{-\calS}$, the L-system $\Theta_{\mu,h}$  has the minimum coefficient of dissipation
$$
\calD_{min}=-\tanh \left(\frac{\calS}{2}\right)\IM m_\infty(i).
$$
In this case, $\Theta_{\mu,h}$ is uniquely determined by the parameters
\begin{equation}\label{e-62-minsys}
    h=-\RE m_\infty(i)-i\tanh \left(\frac{\calS}{2}\right)\IM m_\infty(i)\quad and \quad \mu=\infty.
\end{equation}
\end{theorem}
\begin{proof}
Since $V_{\Theta_{\mu,h}}(z)\in\sM_\kappa$, we have $V_{\Theta_{\mu,h}}(i)=ai$ for some $0<a<1$. Hence by \eqref{e-57-k}, since $0<a<1$, we have   $\kappa=e^{-\calS}=\frac{1-a}{1+a}$. It was shown in \cite{ABT} that there exists a unique Schr\"odinger L-system $\Theta_{\ti\mu,h}$ of the form  \eqref{149} with the same main operator $T_h$ as in $\Theta_{\mu,h}$ such that $V_{\Theta_{\ti\mu,h}}(z)=-1/V_{\Theta_{\mu,h}}(z)$. Indeed, this L-system $\Theta_{\ti\mu,h}$ is determined via \eqref{137} by the parameters $h$ and
\begin{equation}\label{e-63-mu}
\ti\mu=\frac{\mu\,\RE h-|h|^2}{\mu-\RE h}.
\end{equation}
Furthermore, $V_{\Theta_{\ti\mu,h}}(i)=\ti ai=-1/V_{\Theta_{\mu,h}}(i)=i/a$ and hence $\ti a=1/a>1$. Both L-systems $\Theta_{\mu,h}$ and $\Theta_{\ti\mu,h}$ share the same main operator $T_h$ and thus the same c-entropy $\calS$. Repeating the argument in the proof of Theorem \ref{t-10} we obtain
$$
\ti a=\frac{1+\kappa}{1-\kappa}>1\quad \textrm{and }\quad \frac{1}{\ti a}=\frac{1}{1/a}=\frac{e^{\calS}-1}{e^{\calS}+1}=\tanh \left(\frac{\calS}{2}\right).
$$
Applying \eqref{e-59} yields
$$
\IM h=\calD_{min}=-\frac{1}{\ti a}\IM m_\infty(i)=-\left(\frac{e^{\calS}-1}{e^{\calS}+1}\right)\IM m_\infty(i)=-\tanh \left(\frac{\calS}{2}\right)\IM m_\infty(i).
$$
To find $\RE h$ in this case we substitute $\ti \mu=-c=-\RE m_\infty(i)$ in \eqref{e-37-h}. In order to obtain the value of $\mu$ that defines $\Theta_{\mu, h}$ we set $\ti \mu=-c$ in \eqref{e-63-mu} and solve for $\mu$. We have
$$
-c=\frac{\mu\,\RE h-|h|^2}{\mu-\RE h},
$$
yielding
$$
\mu=\frac{c\,\RE h+|h|^2}{c+\RE h}=\infty.
$$
Thus \eqref{e-62-minsys} takes place and the proof is complete.
\end{proof}
We  note that if $a=1$ in the above consideration, then  the c-entropy is infinite, $\calS=+\infty$, and $\kappa=0$ implying
$$\IM h=-\IM m_\infty(i)$$
(see Theorem \ref{t-6}).

As before, the L-system $\Theta_{\mu,h}$ with the given c-entropy and  minimum coefficient of dissipation in the statement of  Theorem  \ref{t-11} can be written explicitly. In particular, its state space operator $\bA$  is described by  formulas \eqref{137} with the values of $h$ given by  \eqref{e-62-minsys}, and $\mu=\infty$ as follows:
\begin{equation}\label{e-64-A-min}
    \begin{split}
&\bA_{\infty,h}\, y=-y^{\prime\prime}+q(x)y-[y^{\prime}(\ell)-h y(\ell)]\, \delta (x-\ell), \\
&\bA^*_{\infty,h}\, y=-y^{\prime\prime}+q(x)y-[y^{\prime}(\ell)-\overline{h}\,y(\ell)]\,\delta(x-\ell),\\
\end{split}
\end{equation}
where $$h=-\RE m_\infty(i)-i\tanh \left(\frac{\calS}{2}\right)\IM m_\infty(i).$$

\section{L-systems with Schr\"odinger's operator and maximal c-entropy}\label{s8}

In this section we set focus on the second part of the dual c-entropy problem that was announced in Introduction, that is:
\begin{itemize}
      \item
\textit{Describe an L-system with Schr\"odinger main operator with a given dissipation coefficient and the \textit{maximal} c-entropy.}
\end{itemize}
Again we are going to  consider  L-systems with Schr\"odinger operator whose impe\-dance functions belong to the generalized Donoghue classes $\sM$, $\sM_\kappa$, and $\sM_\kappa^{-1}$.  Within this class of L-systems we will describe the ones that have a given dissipation coefficient and the \textit{maximal finite} c-entropy.

Let $\Theta_{\mu,h}$ be a Schr\"odinger L-system of the form \eqref{149}. Let also $\kappa$ be the von Neumann parameter of the main operator $T_h$ in $\Theta_{\mu,h}$. Suppose $V_{\Theta_{\mu,h}}(z)$ is the impedance function of \eqref{1501} that belongs to one of the generalized Donoghue classes.  Then, as we have shown in the proof of Theorem \ref{t-7}, the relation $V_{\Theta_{\mu,h}}(i)=ai$ is given by  \eqref{e-43-V} for some $a>0$. Depending on the value of $a$, the parameter $\kappa$ is related to $a$ via \eqref{e-57-k}.

To address the problem of maximizing the c-entropy using a given dissipation coefficient we start at relation \eqref{e-43-V}. Recall the notations we set in \eqref{e-44-set}, that is $c=\RE m_\infty(i)$, $d=\IM m_\infty(i)$, $x=\RE h$, $y=\IM h$.
As we have shown in the proof of Theorem \ref{t-7} relation \eqref{e-43-V} leads to \eqref{e-45} which (after equating real and imaginary parts on both sides of \eqref{e-45}) yields
\begin{equation}\label{e-65-two}
    \begin{aligned}
    (c+\mu)y&=ad(x-\mu),\\
    dy&=ac(\mu-x)+a\mu x-a x^2-a y^2.
    \end{aligned}
\end{equation}
Solving the first equation in \eqref{e-65-two} gives us
\begin{equation}\label{e-66}
    x-\mu=\frac{(c+\mu)y}{ad}\quad\textrm{ and }\quad \mu=\frac{adx-cy}{ad+y}.
\end{equation}
Substituting the expressions from \eqref{e-66} into the second equation \eqref{e-65-two} yields
$$
dy=-ac\cdot\frac{(c+\mu)y}{ad}-ax\cdot\frac{(c+\mu)y}{ad}-ay^2,
$$
or (after cancelation)
$$
d^2=-c(c+\mu)-x(c+\mu)-ady.
$$
Solving  for $a$ results in
$$
a=-\frac{(c+\mu)(c+x)+d^2}{dy}.
$$
Plugging  the expression for $\mu$ from \eqref{e-66} in the above and simplifying leads us to the quadratic equation for $a$
\begin{equation}\label{e-67}
    dy\,a^2+[(c+x)^2+y^2+d^2]\,a+dy=0.
\end{equation}
Notice the discriminant $D$ of the quadratic equation \eqref{e-67} is non-negative and equals zero if only if $c=-x$ and $y=-d$. Indeed,
$$
\begin{aligned}
D&=[(c+x)^2+y^2+d^2]^2-4d^2y^2\\
&=((c+x)^2+y^2+d^2-2dy)((c+x)^2+y^2+d^2+2dy)\\
&=\left((c+x)^2+(y-d)^2\right)\left((c+x)^2+(y+d)^2\right)\ge0.
\end{aligned}
$$
Therefore, if $c\ne -x$ or $y\ne -d$ equation \eqref{e-67} has two distinct real roots
\begin{equation}\label{e-68}
    a_{1,2}=\frac{-(c+x)^2-y^2-d^2\pm\sqrt{[(c+x)^2+y^2+d^2]^2-4d^2y^2}}{2dy}.
\end{equation}
Set
\begin{equation}\label{e-69-nu}
\nu=(c+x)^2+y^2+d^2.
\end{equation}
Our goal is  to show that
$$
a_1=\frac{-\nu+\sqrt{\nu^2-4d^2y^2}}{2dy}<1.
$$
Since $d<0$ and $y>0$, the above inequality is equivalent to
$$
-\nu+\sqrt{\nu^2-4d^2y^2}>2dy\quad\textrm{ or }\quad \nu^2-4d^2y^2>(\nu+2dy)^2,
$$
which transforms to
$$
(\nu+2dy)^2-(\nu+2dy)(\nu-2dy)<0\quad\textrm{ or }\quad 4dy(\nu+2dy)<0.
$$
Taking into account that $4dy<0$, we get
$$
\nu+2dy=(c+x)^2+y^2+d^2=(c+x)^2+(y+2)^2>0,
$$
for $c\ne -x$ or $y\ne -d$. Reversing the argument and following the chain of equivalent inequalities backward we obtain
\begin{equation}\label{e-83-a1}
    a_{1}=\frac{-(c+x)^2-y^2-d^2+\sqrt{[(c+x)^2+y^2+d^2]^2-4d^2y^2}}{2dy}<1,
\end{equation}
for $c\ne -x$ or $y\ne -d$.

Our next goal is to show that
$$
a_2=\frac{-\nu-\sqrt{\nu^2-4d^2y^2}}{2dy}>1,
$$
for $c\ne -x$ or $y\ne -d$. The above inequality is equivalent to
$$
-\nu-\sqrt{\nu^2-4d^2y^2}<2dy,
$$
which transforms to
$$
\sqrt{\nu^2-4d^2y^2}>-\nu-2dy=-(c+x)^2-y^2-d^2-2dy=-(c+x)^2-(y+d)^2.
$$
The last inequality is always true for any $c\ne -x$ or $y\ne -d$ (we are comparing a positive number on the left to a negative number on the right). As before, we pull the chain of equivalent inequalities backward to obtain
\begin{equation}\label{e-69}
    a_{2}=\frac{-(c+x)^2-y^2-d^2-\sqrt{[(c+x)^2+y^2+d^2]^2-4d^2y^2}}{2dy}>1,
\end{equation}
for $c\ne -x$ or $y\ne -d$.

Finally we observe that substituting $c=-x$ and $y=-d$ into  \eqref{e-68} results in
\begin{equation}\label{e-70}
   a=a_1= a_{2}=1,\quad c=-x, \quad y=-d.
\end{equation}

Now we are ready to state the following.
\begin{theorem}\label{t-12}
Assume that an  L-system $\Theta_{\mu,h}$ of the form  \eqref{149} with a  main operator $T_h$ and a coefficient of dissipation $\calD=\IM h\ne-\IM m_\infty(i)$. Then, under the constraint that the impedance function $V_{\Theta_{\mu,h}}(z)\in\sM_\kappa$ for some $0<\kappa<1$, the L-system $\Theta_{\mu,h}$ has the maximum finite c-entropy $\calS_{max}$ if and only if it is  determined by parameters
\begin{equation}\label{e-71-maxsys}
    h=-\RE m_\infty(i)+i\calD\quad and \quad \mu=\left\{
                                                   \begin{array}{ll}
                                                     -\RE m_\infty(i), & \hbox{if $\calD>-\IM m_\infty(i)$;} \\
                                                     \infty, & \hbox{if $\calD<-\IM m_\infty(i)$.}
                                                   \end{array}
                                                 \right.
    \end{equation}
In this case, $\calS_{max}$  is determined by the formula
\begin{equation}\label{e-70-maxent}
\calS_{max}=\ln|\calD-\IM m_\infty(i)|-\ln|\calD+\IM m_\infty(i)|.
\end{equation}
\end{theorem}
\begin{proof}
Since $V_{\Theta_{\mu,h}}(z)\in\sM_\kappa$, then it satisfies a normalization condition
$$V_{\Theta_{\mu,h}}(i)=ai\quad\textrm{ for }\quad a=\frac{1-\kappa}{1+\kappa}.$$ Moreover, since $0<a<1$ in the above, then formula \eqref{e-83-a1} takes place and hence
\begin{equation}\label{e-74}
    a(x)=\frac{-(c+x)^2-\calD^2-d^2+\sqrt{[(c+x)^2+\calD^2+d^2]^2-4d^2\calD^2}}{2d\calD}.
\end{equation}
Taking the derivative of $a(x)$ in \eqref{e-74} and setting it equal to zero yields
$$
\begin{aligned}
a'(x)&=\frac{c+x}{d\calD}\left(\frac{(c+x)^2+\calD^2+d^2}{\sqrt{[(c+x)^2+\calD^2+d^2]^2-4d^2\calD^2}}-1 \right)=\frac{c+x}{d\calD}\cdot\frac{\nu-\sqrt{\nu^2-4d^2\calD^2}}{\sqrt{\nu^2-4d^2\calD^2}}\\
&=(c+x)\cdot\frac{\nu-\sqrt{\nu^2-4d^2\calD^2}}{d\calD\sqrt{\nu^2-4d^2\calD^2}}=0,
\end{aligned}
$$
where $\nu$ is defined in \eqref{e-69-nu}. By inspection,  $$\frac{\nu-\sqrt{\nu^2-4d^2\calD^2}}{d\calD\sqrt{\nu^2-4d^2\calD^2}}<0$$ and that makes the only critical number $x=-c$ to be a point of maximum. Moreover, this maximum value of $a$ is
\begin{equation}\label{e-76-max-a}
    a(-c)=\left\{
                                                   \begin{array}{ll}
                                                     -d/\calD, & \hbox{if $\calD>-d$;} \\
                                                     -\calD/d, & \hbox{if $\calD<-d$.}
                                                   \end{array}
                                                 \right.
\end{equation}
This follows from the direct substitution of $x=-c$ into \eqref{e-74} giving
$$
\begin{aligned}
a(-c)&=\frac{-\calD^2-d^2+\sqrt{[\calD^2+d^2]^2-4d^2\calD^2}}{2d\calD}=\frac{-\calD^2-d^2+\sqrt{[\calD^2-d^2]^2}}{2d\calD}\\
&=\frac{-\calD^2-d^2+|\calD^2-d^2|}{2d\calD}.
\end{aligned}
$$

Suppose that our L-system $\Theta_{\mu,h}$ with a given coefficient of dissipation $\calD=\IM h\ne\IM m_\infty(i)$ has the maximum finite  c-entropy $\calS_{max}$. It follows from the definition of c-entropy and formula \eqref{e-70-entropy} that the maximum finite c-entropy  is achieved when the modulus of von Neumann parameter $\kappa$ is at its  minimum.
Since $0<a<1$, the corresponding $\kappa$ is related to $a$ via \eqref{e-45-kappa-1}. Consider $\kappa$ as a function of a single variable $x$. Then
$$
\kappa(x)=\frac{1-a(x)}{1+a(x)}
$$
is a decreasing function of $a$ and hence $\kappa(x)$ attains its minimum whenever $a(x)$ assumes its maximum. But as we have just shown  above $a(x)$ has its absolute maximum at $x=-c$. Therefore, the L-system $\Theta_{\mu,h}$ has  the maximum c-entropy $\calS_{max}$ when $x=-c=-\RE m_\infty(i)$. The corresponding values of system parameters $h$ and $\mu$ are found once we recall that $x=\RE h$ and $\calD=\IM h$ and then substitute $x=-c=-\RE m_\infty(i)$ in \eqref{e-66}  considering both cases for $a$ in \eqref{e-76-max-a}. This proves the necessity part of the statement.

 Conversely, let the L-system $\Theta_{\mu,h}$ has parameters $h$ and $\mu$ defined by \eqref{e-71-maxsys}. Then as it follows from the first part of \eqref{e-71-maxsys} we have $x=-c=-\RE m_\infty(i)$. As we have shown in the first part of the proof, this value of $x=-c$ maximizes $a(x)$, minimizes $\kappa(x)$, and hence maximizes the c-entropy of our L-system. Thus, $\Theta_{\mu,h}$ has the maximum finite c-entropy $\calS_{max}$ if \eqref{e-71-maxsys} holds.

 In order to prove \eqref{e-70-maxent} we observe that
  \begin{equation}\label{e-77-min-k}
    \kappa(-c)=\left\{
                                                   \begin{array}{ll}
                                                     \frac{\calD+d}{\calD-d}, & \hbox{if $\calD>-d$;} \\
                                                     \frac{\calD+d}{d-\calD}, & \hbox{if $\calD<-d$,}
                                                   \end{array}
                                                 \right.
\end{equation}
is the absolute minimum of $\kappa$ for all real $x$. It follows from the definition of c-entropy and formula \eqref{e-70-entropy} that the minimum of $\kappa$ corresponds to the maximum of c-entropy and we have
$$
\calS(-c)=-\ln|\kappa(-c)|=-\ln\left|\frac{\calD+d}{\calD-d}\right|=\ln|\calD-d|-\ln|\calD+d|,
$$
confirming \eqref{e-70-maxent}.
\end{proof}

A similar result takes place when $a>1$.

\begin{theorem}\label{t-13}
Assume that an L-system $\Theta_{\mu,h}$ of the form  \eqref{149} with a  main operator $T_h$ and a  coefficient of dissipation $\calD=\IM h\ne-\IM m_\infty(i)$. Then, under constraint that the impedance function $V_{\Theta_{\mu,h}}(z)\in\sM_\kappa^{-1}$ for some $0<\kappa<1$,  the L-system $\Theta_{\mu,h}$ has the maximum finite  c-entropy $\calS_{max}$ if and only if it is  determined by parameters
\begin{equation}\label{e-75-maxsys}
    h=-\RE m_\infty(i)+i\calD\quad and \quad \mu=\left\{
                                                   \begin{array}{ll}
                                                   \infty, & \hbox{if $\calD>-\IM m_\infty(i)$;} \\
                                                       -\RE m_\infty(i), & \hbox{if $\calD<-\IM m_\infty(i)$.}
                                                   \end{array}
                                                 \right.
    \end{equation}

In this case, $\calS_{max}$  is determined by \eqref{e-70-maxent}.
\end{theorem}
\begin{proof}
Since $V_{\Theta_{\mu,h}}(z)\in\sM_\kappa^{-1}$,  it satisfies a normalization condition
$$V_{\Theta_{\mu,h}}(i)=ai\quad\textrm{ for }\quad a=\frac{1+\kappa}{1-\kappa}.$$
 Also, since $a>1$ in the above, then formula \eqref{e-69} takes place and hence
\begin{equation}\label{e-79}
    a(x)=\frac{-(c+x)^2-\calD^2-d^2-\sqrt{[(c+x)^2+\calD^2+d^2]^2-4d^2\calD^2}}{2d\calD}.
\end{equation}
Taking the derivative of $a(x)$ in \eqref{e-79} and setting it equal to zero yields
$$
\begin{aligned}
a'(x)&=-\frac{c+x}{d\calD}\left(\frac{(c+x)^2+\calD^2+d^2}{\sqrt{[(c+x)^2+\calD^2+d^2]^2-4d^2\calD^2}}+1 \right)\\
&=-\frac{c+x}{d\calD}\cdot\frac{\nu+\sqrt{\nu^2-4d^2\calD^2}}{\sqrt{\nu^2-4d^2\calD^2}}=-(c+x)\cdot\frac{\nu+\sqrt{\nu^2-4d^2\calD^2}}{d\calD\sqrt{\nu^2-4d^2\calD^2}}=0,
\end{aligned}
$$
where $\nu$ is defined in \eqref{e-69-nu}. Clearly, we have that $$\frac{\nu-\sqrt{\nu^2-4d^2\calD^2}}{d\calD\sqrt{\nu^2-4d^2\calD^2}}>0$$ and $a(x)$ has the only critical number $x=-c$ a point of minimum. Moreover, the minimum value of $a$ is
\begin{equation}\label{e-80-max-a}
    a(-c)=\left\{
                                                   \begin{array}{ll}
                                                   -\calD/d, & \hbox{if $\calD>-d$;} \\
                                                     -d/\calD, & \hbox{if $\calD<-d$.}
                                                   \end{array}
                                                 \right.
\end{equation}
This follows from the direct substitution of $x=-c$ into \eqref{e-74} giving
$$
\begin{aligned}
a(-c)&=\frac{-\calD^2-d^2-\sqrt{[\calD^2+d^2]^2-4d^2\calD^2}}{2d\calD}=\frac{-\calD^2-d^2-\sqrt{[\calD^2-d^2]^2}}{2d\calD}\\
&=\frac{-\calD^2-d^2-|\calD^2-d^2|}{2d\calD}.
\end{aligned}
$$

Suppose that the L-system $\Theta_{\mu,h}$ with a given coefficient of dissipation $\calD=\IM h\ne\IM m_\infty(i)$ has the maximum finite c-entropy $\calS_{max}$. As we already mentioned in the proof of Theorem \ref{t-12},  the maximum finite c-entropy  is achieved when the modulus of  the von Neumann parameter $\kappa$ is at its  minimum.
Since $a>1$, the corresponding $\kappa$ is related to $a$ via \eqref{e-45-kappa-2}. Consider $\kappa$ as a function of a single variable $x$. Then
$$
\kappa(x)=\frac{a(x)-1}{1+a(x)}
$$
 is increasing function of $a$ and hence $\kappa(x)$ takes its minimum whenever $a(x)$ assumes its minimum. But as we have just shown  above $a(x)$ has its absolute minimum at $x=-c$. Therefore, the L-system $\Theta_{\mu,h}$ has  the maximum c-entropy $\calS_{max}$ when $x=-c=-\RE m_\infty(i)$. The corresponding values of system parameters $h$ and $\mu$ are found once we recall that $x=\RE h$ and $\calD=\IM h$ and then substitute $x=-c=-\RE m_\infty(i)$ in \eqref{e-66}  considering both possible cases for $a$ in \eqref{e-80-max-a}. This proves necessity.

 Conversely, let the L-system $\Theta_{\mu,h}$ has the parameters $h$ and $\mu$ defined by \eqref{e-75-maxsys}. Then as it follows from the first part of \eqref{e-75-maxsys} we have $\RE h=x=-\RE m_\infty(i)=-c$. As we have shown in the first part of the proof, this value of $x=-c$ minimizes $a(x)$, minimizes $\kappa(x)$, and hence maximizes c-entropy of our L-system. Thus, $\Theta_{\mu,h}$ has the maximum finite c-entropy $\calS_{max}$ if \eqref{e-75-maxsys} holds.

 In order to prove \eqref{e-70-maxent} we observe that \eqref{e-77-min-k} is still true for this case when one substitutes the values of $a(-c)$ from \eqref{e-80-max-a} into \eqref{e-45-kappa-2}. Hence \eqref{e-77-min-k} again determines  the absolute minimum of $\kappa$ for all real $x$. It follows from  \eqref{e-70-entropy} that the minimum of $\kappa$ corresponds to the maximum of c-entropy which
confirms \eqref{e-70-maxent}.
\end{proof}
We point out that both Theorems \ref{t-12} and \ref{t-13} describe uniquely the L-systems with maximum c-entropy for the corresponding classes of impedance functions. Also, an attentive reader notices that in Theorems \ref{t-12} and \ref{t-13} the hypothesis that
$\calD=\IM h\ne-\IM m_\infty(i)$ holds and might wonder what happens if  $ h=- m_\infty(i)$ for a given L-system. The following remark addresses that question.

\begin{remark}\label{t-14}
An L-system $\Theta_{\mu,h}$ of the form  \eqref{149} with a  Schr\"odinger operator $T_h$ and a parameter $\mu\in\Bbb R\cup \{\infty\}$ has  infinite c-entropy,  $\calS=\infty$ if and only if $h=-m_\infty(i)$.
Indeed, one might also notice that if $h=-m_\infty(i)$, then $\calD=\IM h=-\IM m_\infty(i)$ and hence formulas \eqref{e-76-max-a} and \eqref{e-80-max-a}  imply that $a(-c)=1$ and $\kappa(-c)=0$. The assertion then is a corollary of Theorem \ref{t-6}.
\end{remark}

\section{c-Entropy in sectorial cases and the Krein-von Neumann extension}\label{s9}


Throughout this section we deal with L-systems whose main operators are either extremal accretive or $\beta$-sectorial.

\subsection{Dual c-entropy problems in the extremal case}

The following theorem gives the solution to the first dual c-entropy problem posed in Section \ref{s7}.
\begin{theorem}\label{t-20}
Assume that  an L-system  $\Theta_{\mu,h}$ of the form  \eqref{149} with a  main operator $T_h$ has c-entropy $\calS$.  Then, under the constraint  that $T_h$ is extremal, the L-system $\Theta_{\mu,h}$  has the minimum coefficient of dissipation
\begin{equation}\label{e-96-dis}
\calD_{min}=-\IM m_\infty(i)\coth\calS-\sqrt{B^2\,\mathrm{csch^2}\,\calS-(A-m)^2},
\end{equation}
where  $A=\RE m_\infty(i)$, $B=-\IM m_\infty(i)$,  and $m=m_\infty(-0)$.

In this case, the L-system $\Theta_{\mu,h}$ is  determined by the parameters
\begin{equation}\label{e-97-minsys}
    h=-m_\infty(-0)+i\calD_{min}\quad and\ \  arbitrary \quad \mu\in\Bbb R\cup \{\infty\}.
\end{equation}
\end{theorem}
\begin{proof}
Suppose that the Schr\"odinger L-system $\Theta_{\mu,h}$ has a given fixed c-entropy $\calS$.
Then, as we have shown in Section \ref{s5}, the modulus of von Neumann parameter $\kappa$ is related to the coefficient of dissipation $\calD=\IM h$ via \eqref{e-45-new}
\begin{equation}\label{e-96-kappa-D}
\kappa^2=\frac{\calD^2-2B \calD+D}{\calD^2+2B \calD+D},
\end{equation}
where $B=-\IM m_\infty(i)$, $D=C+B^2$, and $C=(\RE m_\infty(i)-m_\infty(-0))^2$. By Theorem \ref{t-9-1},  $\kappa_0\le\kappa<1$, where $\kappa_0$ is given by \eqref{e-47}. Since $\kappa=e^{-\calS}$, by \eqref{e-96-kappa-D}, we have
$$
e^{-2\calS}=\frac{\calD^2-2B \calD+D}{\calD^2+2B \calD+D}
$$
and hence
$$
(1-e^{-2\calS})\calD^2-2B(1+e^{-2\calS})\calD+(1-e^{-2\calS})D=0,
$$
or
\begin{equation}\label{e-97-eq}
   \calD^2-2B(\coth\calS)\calD+D=0.
\end{equation}
This equation (see the proof of Theorem \ref{t-9-2})  has either two or one real roots (also see Figure \ref{fig-1} for convenience) given by
\begin{equation}\label{e-98-min}
    \calD=B\coth\calS\pm\sqrt{B^2\coth^2\calS-D},
\end{equation}
where $B$ and $D$ are defined above. The case of a unique solution occurs when $e^{-\calS}=\kappa=\kappa_0$. As we have shown in Section \ref{s5}, in this case $$\calD=\sqrt{D}=\sqrt{(\IM m_\infty(i)-m_\infty(-0))^2+(\IM m_\infty(i))^2}.$$ If $\kappa_0<\kappa<1$, then the minimal coefficient of dissipation is the smallest of two  roots of equation \eqref{e-98-min}. It is easy to see that
$$
\calD_{min}=B\coth\calS-\sqrt{B^2\coth^2\calS-D}
$$
in this case. Substituting the values for $B$ and $D$ into the above formula and taking into account that $\mathrm{csch}^2\,\calS=\coth^2\calS-1$ we obtain \eqref{e-96-dis}.
\end{proof}


Now we address the second dual c-entropy problem.  
\begin{theorem}\label{t-21}
Assume that an  L-system $\Theta_{\mu,h}$ is of the form  \eqref{149} with a  main operator $T_h$. Assume in addition that $\RE m_\infty(i)\ne m_\infty(-0)$.
Then, under the constraint that $T_h$ is extremal, the L-system $\Theta_{\mu,h}$ has the maximum finite c-entropy $\calS_{max}$ if and only if it is determined by the parameters
\begin{equation}\label{e-115-maxent}
    h=-m+i\sqrt{(A-m)^2+B^2}\quad and\ \  (arbitrary) \quad \mu\in\Bbb R\cup \{\infty\},
\end{equation}
where  $A=\RE m_\infty(i)$, $B=-\IM m_\infty(i)$, and $m=m_\infty(-0)$.

In this case,
\begin{equation}\label{e-100-maxent}
\calS_{max}=-\frac{1}{2}\ln\left( \frac{\sqrt{A^2+B^2-2mA+m^2}-B}
{\sqrt{A^2+B^2-2mA+m^2}+B}\right).
\end{equation}
\end{theorem}
\begin{proof}

It follows from the development in Section \ref{s5}  that there are many  L-system of the form \eqref{149}  with an extremal main operators $T_h$ only different by the coefficient of dissipation (see also Figure \ref{fig-1}).
Naturally, such an L-system has the maximum c-entropy $\calS_{max}$ if the modulus of von Neumann parameter $\kappa$ of $T_h$ is at minimum, i.e., $\kappa=\kappa_0$ (see formula \eqref{e-47}). Note that $\kappa_0\ne0$ or otherwise $A=m$ and our L-system $\Theta_{\mu,h}$ has infinite entropy according to Theorem \ref{t-6}. We know that $\kappa=\kappa_0$   if and only if
\begin{equation}\label{e-90-imh}
\IM h=\sqrt{(\RE m_\infty(i)-m_\infty(-0))^2+(\IM m_\infty(i))^2}=\sqrt{(A-m)^2+B^2}.
\end{equation}
In this case (see \eqref{e-70-entropy}) $\calS_{max}=-\ln(\kappa_0)$ proving \eqref{e-100-maxent}.
\end{proof}
\begin{remark}\label{r-26}
{If  $\RE m_\infty(i)=m_\infty(-0)$, then $h=-m_\infty(i)$ and our L-system $\Theta_{\mu,h}$ has infinite entropy according to Theorem \ref{t-6}.}
\end{remark}


\subsection{Dual c-entropy problems in the sectorial case}\label{ss9-2}
We use similar analysis to treat the sectorial case. We remind the reader that by saying that an accretive operator is $\beta$-sectorial, we mean that $\beta\in(0,\pi/2)$ is its exact angle of sectoriality unless otherwise is specified.

 The following theorem is analogues to Theorem \ref{t-20} for the sectorial case. 

\begin{theorem}\label{t-22}
Assume that an L-system $\Theta_{\mu,h}$ of the form  \eqref{149} with a  main operator $T_h$ has c-entropy $\calS$. Then, under the constraint that $T_h$ is  $\beta$-sectorial, the L-system $\Theta_{\mu,h}$  has the minimum coefficient of dissipation
\begin{equation}\label{e-102-dis}
\calD_{min}=\sin^2\beta\left[B\coth \calS-E\cot\beta-\sqrt{(E\cot\beta-B\coth \calS)^2-D} \right],
\end{equation}
where $A=\RE m_\infty(i)$, $B=-\IM m_\infty(i)$, $D=E^2+B^2$, $E=A-m$, 
and $m=m_\infty(-0)$.

In this case, the L-system $\Theta_{\mu,h}$ is  determined by the parameters
\begin{equation}\label{e-119-minsys}
    h=(\cot\beta)\calD_{min}-m+i\calD_{min}\quad and\ \  (arbitrary) \quad \mu\in\Bbb R\cup \{\infty\}.
\end{equation}
\end{theorem}
\begin{proof}

Let $\Theta_{\mu,h}$ have a $\beta$-sectorial main operator $T_h$ whose von Neumann's parameter is $\kappa$. Then $\kappa$ is related to the coefficient of dissipation $\calD=\IM h$ via \eqref{e-49}, that is
\begin{equation}\label{e-99-kappa-D}
\kappa^2=\frac{\calD^2+2\sin^2\beta((A-m)\cot\beta-B)\calD+D\sin^2\beta}{\calD^2+2\sin^2\beta((A-m)\cot\beta+B)\calD+D\sin^2\beta}.
\end{equation}
According to Theorem \ref{t-10-1}, the parameter $\kappa$ in this case belongs to the interval $0<\kappa_0\le\kappa<1$, where $\kappa_0$ is given by \eqref{e-50}, \eqref{e-51}. Note that $\kappa_0\ne0$ since otherwise this will contradict Lemma \ref{l-2}.

 In order to find the minimal coefficient of dissipation we recall that $\kappa=e^{-\calS}$, use \eqref{e-99-kappa-D}, and solve
$$
e^{-2\calS}=\frac{\calD^2+2\sin^2\beta((A-m)\cot\beta-B)\calD+D\sin^2\beta}{\calD^2+2\sin^2\beta((A-m)\cot\beta+B)\calD+D\sin^2\beta}
$$
for $\calD$ to pick the minimal root. The above equation simplifies to the quadratic one of the form
$$\begin{aligned}
(1-e^{-2\calS})\calD^2&+2B\sin^2\beta\left[(1-e^{-2\calS})(A-m)-\left(1+e^{-2\calS}\right)B\right]\calD\\
&+(1-e^{-2\calS})D\sin^2\beta=0,
\end{aligned}
$$
or, with $E=A-m$,
\begin{equation}\label{e-100-eq}
    \calD^2+2B\sin^2\beta\left[E\cot\beta-\left(\coth\calS\right)B\right]\calD+D\sin^2\beta=0.
\end{equation}
Equation \eqref{e-100-eq} has either two or one real solutions (see also Figure \ref{fig-2}) given by
\begin{equation}\label{e-101-min}
    \calD=(\sin^2\beta)\left[B\coth \calS-E\cot\beta\pm\sqrt{(E\cot\beta-B\coth \calS)^2-D} \right].
\end{equation}
 The case of a unique solution occurs when $e^{-\calS}=\kappa=\kappa_0$, where $\kappa_0$ is given by \eqref{e-50}, \eqref{e-51}. As we have shown in Section \ref{s6}, in this case
$$\calD=(\sin\beta)\sqrt{D}=(\sin\beta)\sqrt{|m_\infty(i)|^2-2m_\infty(-0)\RE\,m_\infty(i)+m_\infty^2(-0)}.$$
If $\kappa_0<\kappa<1$, then the minimal coefficient of dissipation is the smallest of the two  roots of \eqref{e-101-min}. It is easy to see that in this case
\begin{equation}\label{e-103-min}
\calD_{min}=(\sin\beta)^2\left[B\coth \calS-E\cot\beta-\sqrt{(E\cot\beta-B\coth \calS)^2-D} \right].
\end{equation}
Formula \eqref{e-119-minsys} follows from the relation \eqref{e-48} connecting the real and imaginary part of the parameter $h$ that describes a $\beta$-sectorial operator $T_h$.
\end{proof}

Now we address the second dual problem of maximizing c-entropy for the class of Schr\"odinger L-systems with sectorial main operators.

The following theorem is similar to Theorem \ref{t-21}.
\begin{theorem}\label{t-23}
Assume that an  L-system $\Theta_{\mu,h}$ is of the form  \eqref{149} with a  main operator $T_h$. Then, under the constraint that $T_h$ is $\beta$-sectorial, the L-system $\Theta_{\mu,h}$ has the maximum finite c-entropy $\calS_{max}$ if and only if it is determined by the parameters
\begin{equation}\label{e-126-maxent}
    h=(\cos\beta)\sqrt{D}-m+i(\sin\beta)\sqrt{D}\quad and\ \  (arbitrary) \quad \mu\in\Bbb R\cup \{\infty\},
\end{equation}
where  $A=\RE m_\infty(i)$, $B=-\IM m_\infty(i)$, $D=(A-m)^2+B^2$,  and $m=m_\infty(-0)$.

In this case,
\begin{equation}\label{e-95-Smax}
    \calS_{max}=-\frac{1}{2}\ln\left( \frac{\sqrt D+(A-m)\cos\beta-B\sin\beta}{\sqrt D+(A-m)\cos\beta+B\sin\beta}\right).
\end{equation}
\end{theorem}
\begin{proof}

It follows from the development in  Section \ref{s6} and formulas \eqref{e-50} and \eqref{e-51}  that  an L-system of the form \eqref{149} with a $\beta$-sectorial main operator $T_h$ has the maximum c-entropy $\calS_{max}$ if the modulus of von Neumann parameter $\kappa$ of $T_h$ is at minimum (see also Figure \ref{fig-2}). That happens if and only if $\kappa=\kappa_0$ (see  \eqref{e-50} and  \eqref{e-51}) and hence
\begin{equation}\label{e-102-D}
\IM h=\sin\beta\cdot\sqrt{D}=\sin\beta\cdot\sqrt{(\RE m_\infty(i)-m_\infty(-0))^2+(\IM m_\infty(i))^2}.
\end{equation}
Note that $\kappa_0\ne0$ or otherwise this will violate Lemma \ref{l-2}.
Moreover, in this case (see \eqref{e-70-entropy}) $\calS_{max}=-\ln(\kappa_0)$ proving \eqref{e-95-Smax}. The expression in the right hand side of \eqref{e-95-Smax} is well defined as the fraction inside the logarithm represents $\kappa_0^2$, the positive quantity.
\end{proof}

\subsection{Accretive $T_h$ case}
In this subsection we are going to look at the situation when the main operator $T_h$ of the Schr\"odinger L-system under consideration is just accretive. Recall  (see Section \ref{s2} for the definition) that the operator $T_h$ is accretive if $\RE(T_h y,y)\ge 0$ for all $y\in \dom(T_h)$. Clearly, the set of all such accretive operators consists of the  class of $\beta$-sectorial (for some $\beta\in(0,\pi/2))$ operators plus the class of extremal operators. 

The first question we address is  which  L-systems $\Theta_{\mu,h}$ with accretive main operators $T_h$ have the maximal finite c-entropy.
\begin{theorem}\label{t-24}
Assume that an  L-system $\Theta_{\mu,h}$ is of the form  \eqref{149} with a  main operator $T_h$. Then, under the constraint that $T_h$ is accretive,   the L-system $\Theta_{\mu,h}$ with extremal accretive main operator  achieves the maximum finite c-entropy $\calS_{max}$.

 In this case, $\calS_{max}$ is given by \eqref{e-100-maxent} and $\Theta_{\mu,h}$ is determined by the parameters
\begin{equation}\label{e-109-h}
h=-m+i\sqrt{(A-m)^2+B^2}\quad and \quad (arbitrary)\; \mu\in\Bbb R\cup \{\infty\},
\end{equation}
where $A=\RE m_\infty(i)$, $B=-\IM m_\infty(i)$,  and $m=m_\infty(-0)$.
\end{theorem}
\begin{proof}
We have already mentioned above that since our main operator $T_h$ is accretive, then it is either $\beta$-sectorial (for some $\beta\in(0,\pi/2))$ or extremal accretive. In the first case the  L-system $\Theta_{\mu,h}$ with $\beta$-sectorial main operator has the maximum entropy $\calS_{max}^{sec}$ given by \eqref{e-95-Smax}. In the second case the  L-system $\Theta_{\mu,h}$ with the extremal main operator has the maximum entropy $\calS_{max}^{ext}$ given by \eqref{e-100-maxent}.

In order to prove the theorem we need to confirm that the latter is larger, that is
$$
\calS_{max}^{sec} =-\frac{1}{2}\ln\left(\frac{\sqrt D+(A-m)\cos\beta-B\sin\beta}{\sqrt D+(A-m)\cos\beta+B\sin\beta}\right)< \calS_{max}^{ext}=-\frac{1}{2}\ln\left(\frac{\sqrt{D}-B}{\sqrt{D}+B}\right),
$$
for any $\beta\in(0,\pi/2)$. Here $D=(A-m)^2+B^2$. This inequality is clearly equivalent to the simpler one
\begin{equation}\label{e-128-in}
   \frac{\sqrt D+(A-m)\cos\beta-B\sin\beta}{\sqrt D+(A-m)\cos\beta+B\sin\beta}> \frac{\sqrt{D}-B}{\sqrt{D}+B},\quad \forall \beta\in(0,\pi/2).
\end{equation}
We bring attention of the reader to the parameters $A$ and $m$ in the inequality \eqref{e-128-in}.  As we have shown in Lemma \ref{l-16}, $A\ge m$. If  we assume that $A= m$, then it follows from Lemma \ref{l-13} that the minimal von Neumann's parameter $\kappa_0$ associated with the extremal operator $T_h$ is such that $\kappa_0=0$ and hence the maximum c-entropy of L-system with such extremal $T_h$ is infinite, i.e., $\calS_{max}^{ext}=\infty$. In this theorem we are only interested in the case of finite  maximum c-entropy and thus the case when $A= m$ should not be considered. Thus, we can assume without loss of generality that $A>m$.


Consider the function
\begin{equation}\label{e-104-f}
    f(\beta)=\frac{\sqrt D+(A-m)\cos\beta-B\sin\beta}{\sqrt D+(A-m)\cos\beta+B\sin\beta}\quad\textrm{ on }\quad \beta\in(0,\pi/2).
\end{equation}
Since $A> m$, we have that  the derivative
\begin{equation}\label{e-105-f-prime}
    f'(\beta)=-\frac{2B\left(A-m+\sqrt D\cdot\cos\beta\right)}{\left(\sqrt D+(A-m)\cos\beta+B\sin\beta\right)^2}<0
\end{equation}
 for any   $\beta\in(0,\pi/2)$. Hence the function $f(\beta)$ in \eqref{e-104-f} decreases on $\beta\in(0,\pi/2)$ and 
$$
f\left(\frac{\pi}{2}-\right)=\frac{\sqrt{D}-B}{\sqrt{D}+B}.
$$
Note that  under our current assumption $f(\pi/2-)\ne0$ or otherwise $A=m$ (see Lemma \ref{l-13}).

Therefore, inequality \eqref{e-128-in} holds for $A>m$.
Consequently, a  Schr\"odinger L-system $\Theta_{\mu,h}$  with an accretive main operator $T_h$  achieves the maximum finite c-entropy whenever $T_h$ is extremal. In this case, according to Theorem \ref{t-21}, $\calS_{max}$ is determined by \eqref{e-100-maxent} while $h$ is given by \eqref{e-109-h} and $\mu$ is arbitrary in $\Bbb R\cup \{\infty\}$.
\end{proof}

\begin{remark}\label{r-30}

Note that if $\Theta_{\mu,h}$ is the  Schr\"odinger L-system referred to in Theorem \ref{t-24}, then the modulus of von Neumann's parameter of the corresponding main operator $T_h$ equals $\kappa_0$  given by \eqref{e-47}.
Moreover, if under the assumptions of Theorem \ref{t-24}, $\Theta_{\mu_1,h}$ and $\Theta_{\mu_2,h}$ are such L-systems that the corresponding impedance functions $V_{\Theta_{\mu_1,h}}(z)$ and $V_{\Theta_{\mu_2,h}}(z)$ belong to he generalized Donoghue classes $\sM_{\kappa_0}$ or $\sM_{\kappa_0}^{-1}$ respectively, then
 \begin{equation}\label{e-112-mu1}
    \mu_1=\frac{(A-mBF)(BF-1)+(m^2F+DF-B-mAF)(A-m)F}{(BF-1)^2+(A-m)^2F^2}
\end{equation}
and
\begin{equation}\label{e-114-mu2}
    \mu_2=-\frac{\mu_1 m+m^2+(A-m)^2+B^2}{\mu_1+m}.
\end{equation}
Here $$F=\frac{\sqrt{D}-\sqrt{C}}{B\sqrt{D}}$$ and all the other letters are defined above.

Indeed, let the L-systems $\Theta_{\mu_1,h}$ and $\Theta_{\mu_2,h}$  be such that
\begin{equation}\label{e-110-mu}
    V_{\Theta_{\mu_1,h}}(i)=ai\quad\textrm{ and}\quad V_{\Theta_{\mu_2,h}}(i)=\left(\frac{1}{a}\right)i,
\end{equation}
where
$$
a=\frac{1-\kappa_0}{1+\kappa_0}=\frac{1-\sqrt{\frac{\sqrt D-B}{\sqrt D+B}}}{1+\sqrt{\frac{\sqrt D-B}{\sqrt D+B}}}=\frac{\sqrt{D}-\sqrt{D-B^2}}{B}=\frac{\sqrt{D}-\sqrt{C}}{B}.
$$
Then \eqref{e-110-mu} will guarantee that  for the corresponding impedance functions  we have $V_{\Theta_{\mu_1,h}}(z)\in\sM_{\kappa_0}$ and $V_{\Theta_{\mu_2,h}}(z)\in\sM_{\kappa_0}^{-1}$.   By \eqref{1501}
$$
V_{\Theta_{\mu,h}}(i)=\frac{(A-Bi+\mu)\sqrt{D}}{(\mu+m)(A-Bi)-\mu m-m^2-D}.
$$
In view of \eqref{e-110-mu} we have the equation
$$
\frac{(A-Bi+\mu)\sqrt{D}}{(\mu+m)(A-Bi)-\mu m-m^2-D}=\frac{\sqrt{D}-\sqrt{C}}{B}i,
$$
which we are going to solve for $\mu_1$. To simplify the calculation process we set $F=a/\sqrt{D}=\frac{\sqrt{D}-\sqrt{C}}{B\sqrt{D}}$ and hence
\begin{equation}\label{e-111-mu}
\begin{aligned}
\mu_1&=\frac{A-mBF+(m^2F+DF-B-mAF)i}{BF-1+(A-m)Fi}\\
&=\frac{(A-mBF)(BF-1)+(m^2F+DF-B-mAF)(A-m)F}{(BF-1)^2+(A-m)^2F^2}\\
&\quad+i\frac{(m^2F+DF-B-mAF)(BF-1)-(A-mBF)(A-m)F}{(BF-1)^2+(A-m)^2F^2}.
\end{aligned}
\end{equation}
Substituting $D=(A-m)^2+B^2$ and $F=\frac{\sqrt{D}-\sqrt{C}}{B\sqrt{D}}$ we confirm that the imaginary part of $\mu_1$ in \eqref{e-111-mu} is zero. Thus,
\begin{equation}\label{e-112-mu1-1}
    \mu_1=\frac{(A-mBF)(BF-1)+(m^2F+DF-B-mAF)(A-m)F}{(BF-1)^2+(A-m)^2F^2}.
\end{equation}

To find $\mu_2$ we use \eqref{e-63-mu} and the reasoning in the proof of Theorem \ref{t-11} (see also \cite{ABT}) and get
\begin{equation}\label{e-114-mu2-2}
    \mu_2=\frac{\mu_1\RE h-|h|^2}{\mu_1-\RE h}=-\frac{\mu_1 m+m^2+(A-m)^2+B^2}{\mu_1+m}.
\end{equation}
Thus, the Schr\"odinger L-systems $\Theta_{\mu_1,h}$ and $\Theta_{\mu_2,h}$ with the desired properties  are uniquely constructed using the values of $h$ from \eqref{e-109-h},  $\mu_1$ from \eqref{e-112-mu1}, and $\mu_2$ from \eqref{e-114-mu2}.
\end{remark}

\subsection{Maximal c-entropy in accretive case} In this subsection we consider Schr\"odinger L-systems with maximal finite c-entropy and accretive state-space operator $\bA$. The following theorem describes an L-system with  $\beta$-sectorial main and state-space operators and maximal c-entropy.
\begin{theorem}\label{t-27}
An  L-system $\Theta_{\mu,h}$ of the form  \eqref{149} with a $\beta$-sectorial ($\beta\in(0,\pi/2)$) Schr\"odinger main operator $T_h$  and the maximum finite c-entropy $\calS_{max}$ determined by the formula \eqref{e-95-Smax} has a $\beta$-sectorial (with the same angle of sectoriality) state-space operator $\bA_{\mu,h}$ if and only if $\mu=\infty$ in \eqref{137}.
\end{theorem}
\begin{proof}
It was shown in \cite[Theorem 10.6.4]{ABT} (see also \cite{B2011}, \cite{BT-15}) that if $\bA_{\mu,h}$ is a ($*$)-extension of an $\beta$-sectorial operator $T_h$
with the exact angle of sectoriality $\beta\in(0,\pi/2)$, then $\bA_{\mu,h}$ is an $\beta$-sectorial ($*$)-extension of $T_h$ (with the same angle of sectoriality) if and only if $\mu=+\infty$. The rest follows from Theorem \ref{t-23}.
\end{proof}

\begin{theorem}\label{t-28}
An  L-system $\Theta_{\mu,h}$ of the form  \eqref{149} with a $\beta$-sectorial ($\beta\in(0,\pi/2)$) Schr\"odinger main operator $T_h$  and the maximum finite c-entropy $\calS_{max}$ determined by the formula \eqref{e-95-Smax} has an extremal state-space operator $\bA_{\mu,h}$ if and only if
\begin{equation}\label{e10-134}
\mu=(2\csc2\beta)\sqrt{(\RE m_\infty(i)-m_\infty(-0))^2+(\IM m_\infty(i))^2}-m_\infty(-0).
\end{equation}
\end{theorem}
\begin{proof}
It follows from \cite[Theorem 10.6.5]{ABT} (see also \cite{B2011}, \cite{BT-15})  that if $\bA_{\mu,h}$ is a ($*$)-extension of an $\beta$-sectorial operator $T_h$
with the exact angle of sectoriality $\beta\in(0,\pi/2)$, then $\bA_{\mu,h}$ is accretive but not  $\beta$-sectorial for any $\beta\in(0,\pi/2)$ ($*$)-extension of $T_h$  if and only if in \eqref{137}
\begin{equation}\label{e-117-mu}
\mu=\frac{(\IM h)^2}{m_\infty(-0)+\RE h}+\RE h.
\end{equation}
 Taking into account that in our case $\IM h$ and $\RE h$ are related by \eqref{e-48}, we substitute their values into \eqref{e-117-mu} to get 
$$
\begin{aligned}
\mu&=\frac{(\IM h)^2}{m+\RE h}+\RE h=\frac{(\IM h)^2}{m+(\cot\beta)\IM h-m}+(\cot\beta)\IM h-m\\
&=(\tan\beta)\IM h+(\cot\beta)\IM h-m=(\tan\beta+\cot\beta)\IM h-m\\
&=(2\csc2\beta)\IM h-m,
\end{aligned}
$$
where $m=m_\infty(-0)$. We know from Theorem \ref{t-23} that if the L-system $\Theta_{\mu,h}$ has the maximum c-entropy, then $\IM h$ is given by \eqref{e-102-D}. Substituting the value of $\IM h$ into the above we obtain \eqref{e10-134}.
\end{proof}
Now we focus on the case when  the main operator of a Schr\"odinger L-systen is extremal.
\begin{theorem}\label{t-26}
An  L-system $\Theta_{\mu,h}$ of the form  \eqref{149} with an extremal Schr\"odinger main operator $T_h$  and the maximum finite c-entropy $\calS_{max}$ determined by the formula \eqref{e-100-maxent} has an accretive state-space operator $\bA_{\mu,h}$ if and only if $\mu=\infty$.
\end{theorem}
\begin{proof}
As it was shown in \cite[Theorem 10.6.4]{ABT} (see also \cite{B2011}, \cite{BT-15}), a state-space operator $\bA_{\mu,h}$ preserves the exact angle of sectoriality of $T_h$ if and only if $\mu=\infty$. Here our operator $T_h$ is extremal  and thus can only admit one extremal $(*)$-extension $\bA_{\mu,h}$ with $\mu=\infty$.
\end{proof}

At this point again we  are going to look at the combined class of $\beta$-sectorial and extremal accretive operators to see what case provides us with an L-system with maximal finite c-entropy.
\begin{theorem}\label{t-29}
An  L-system $\Theta_{\mu,h}$ of the form  \eqref{149} with an accretive Schr\"odinger main operator $T_h$  and accretive state-space operator $\bA_{\mu,h}$  achieves the maximum finite c-entropy $\calS_{max}$ determined by \eqref{e-100-maxent} when $T_h$ is extremal.

In this case $h$ is given by \eqref{e-109-h}, $\mu=\infty$, and the quasi-kernel of $\RE\bA_{\mu,h}$ is the Krein-von Neumann extension defined by the formula
\begin{equation}
\label{200} \left\{\begin{array}{l}
 A_Ky=-y^{\prime\prime}+q(x)y, \\
 y^{\prime}(\ell)+m_\infty(-0)y(\ell)=0.
 \end{array}\right.
\end{equation}
\end{theorem}
\begin{proof}
We know that according to Theorem \ref{t-24} among all the  L-system $\Theta_{\mu,h}$ of the form  \eqref{149} with an accretive Schr\"odinger main operator $T_h$  and any $\mu\in\Bbb R\cup \{\infty\}$, the L-system that has the maximum  finite c-entropy is the one with  the extremal main operator $T_h$. On the other hand, since we are given an extra condition that the state-space operator $\bA_{\mu,h}$ is also accretive, then Theorem \ref{t-26} provides us with $\mu=\infty$ for this case. Consequently, $\bA_{\mu,h}$ is given by \eqref{e-64-A-min} that is
\begin{equation}\label{e-119}
    \begin{split}
&\bA_{\mu,h}=\bA_{\infty,h}\, y=-y^{\prime\prime}+q(x)y-[y^{\prime}(\ell)-h y(\ell)]\, \delta (x-\ell), \\
&\bA^*_{\mu,h}=\bA^*_{\infty,h}\, y=-y^{\prime\prime}+q(x)y-[y^{\prime}(\ell)-\overline{h}\,y(\ell)]\,\delta(x-\ell).\\
\end{split}
\end{equation}
Using \eqref{e-119} one easily gets
\begin{equation}\label{e-120}
\RE\bA_{\infty,h}=-y^{\prime\prime}+q(x)y+\left[y^{\prime}(\ell)+m_\infty(-0)y(\ell)\right]\delta(x-\ell).
\end{equation}
Clearly, \eqref{e-120} implies that  $\RE\bA_{\infty,h} \supset A_K$ and $A_K$ given by \eqref{200} is the quasi-kernel for $\RE\bA_{\infty,h}$.
\end{proof}

\begin{remark}\label{r-29}
Note that if an  L-system $\Theta_{\mu,h}$ of the form  \eqref{149} has an accretive state-space operator, then its impedance function $V_{\Theta_{\mu,h}}(z)$ cannot belong to any of the generalized Donoghue classes.

Indeed, we know from \cite[Theorem 9.8.2]{ABT} that if an L-system has an accretive state-space operator, then its impedance function is a Stieltjes function and thus has the integral representation
\begin{equation}\label{e8-94}
V_{\Theta_{\mu,h}}(z) =\gamma+\int\limits_0^\infty\frac {d\sigma(t)}{t-z},
\end{equation}
where $\gamma\ge0$ and $\sigma(t)$ is a non-decreasing on $[0,+\infty)$  function such that $\int^\infty_0\frac{dG(t)}{1+t}<\infty.$ On the other hand, if we assume that a Stieltjes function $V_{\Theta_{\mu,h}}(z)$ belonged to a generalized Donoghue class, then it would have admited integral representation of the form \eqref{hernev-0}, that is \begin{equation}\label{hernev-119}
V_{\Theta_{\mu,h}}(z)=\int_0^{+\infty} \left(\frac{1}{t-z}-\frac{t}{1+t^2}\right)d\sigma(t).
\end{equation}
However, calculating the real part of $V_{\Theta_{\mu,h}}(i)$ using both \eqref{e8-94} and \eqref{hernev-119} and then  setting the results equal  leads to
$$
\RE V_{\Theta_{\mu,h}}(i)=0=\gamma+\int\limits_0^\infty\frac {td\sigma(t)}{t^2+1},
$$
which is a contradiction since $\gamma\ge0$ and $\int_0^\infty\frac {td\sigma(t)}{t^2+1}>0$ (see \cite{ABT}). Therefore we have shown that a Stieltjes function $V_{\Theta_{\mu,h}}(z)$ can be written as
$$
V_{\Theta_{\mu,h}}(z)=Q+\int_0^{+\infty} \left(\frac{1}{t-z}-\frac{t}{1+t^2}\right)d\sigma(t),
$$
with
\begin{equation}\label{e-135-Q}
    Q=\gamma+\int\limits_0^\infty\frac {td\sigma(t)}{t^2+1}>0.
\end{equation}
Consequently, none of the Stieltjes impedance functions $V_{\Theta_{\mu,h}}(z)$ can belong to any of the generalized Donoghue classes.
\end{remark}

\section{Table of dual c-entropy problems solutions}
In this subsection we  summarize the results on dual c-entropy optimization problems and present them in the tabular form (see Table \ref{Table-1}).
\begin{table}[ht]
\centering
\begin{tabular}{|c|c|c|c|}
\hline
 &  &  &\\
 L-system& First dual  & Second dual  & L-system \\
  property& problem & problem &uniqueness\\
   &  &  &\\
  \hhline{|=|=|=|=|}
&  &  &\\
  $V_\Theta(z)\in\sM_\kappa$ & Theorem \ref{t-11} & Theorem \ref{t-12} & Unique\\
  &  &  &\\  \hline
 &  &  &\\
  $V_\Theta(z)\in\sM_\kappa^{-1}$ & Theorem \ref{t-10} & Theorem \ref{t-13} & Unique\\
  &  &  &\\  \hline
&  &  &\\
  $T_h$ is extremal & Theorem \ref{t-20} & Theorem \ref{t-21} & Not unique\\
  &  &  &\\  \hline
 &  &  &\\
  $T_h$ is $\beta$-sectorial  & Theorem \ref{t-22} & Theorem \ref{t-23} & Not unique\\
  &  &  &\\  \hline
     \multicolumn{1}{l}{} & \multicolumn{1}{l}{} & \multicolumn{1}{l}{}& \multicolumn{1}{l}{}
\end{tabular}
\caption{Solutions to dual c-entropy problems}
\label{Table-1}
\end{table}

 In the very left column of the table we list the properties of the Schr\"odinger L-systems under consideration.
The second and the third columns of the table indicate the theorems that contain solutions to the first and the second dual c-entropy problems, respectively. The last column of Table \ref{Table-1} addresses the uniqueness of the Schr\"odinger L-system solving either of the dual c-entropy problems. We have already noted in Sections \ref{s7} and \ref{s8} that the Schr\"odinger L-systems $\Theta_{\mu,h}$ referred to in Theorems \ref{t-10}--\ref{t-13} are unique with the exact values of the defining parameters $\mu$ and $h$ provided. On the other hand, when the dual c-entropy problems are being solved for the classes of Schr\"odinger L-systems with extremal or $\beta$-sectorial main operators $T_h$ the answer is not unique. Namely, Theorems \ref{t-20}--\ref{t-23} describe  the entire family of solutions. These solutions $\Theta_{\mu,h}$ are parameterized by a particular value of $h$ and an arbitrary parameter $\mu\in\Bbb R\cup \{\infty\}$. Note that the case when the impedance function $V_\Theta(z)$ belongs to the Donoghue class $\sM$ deals with infinite c-entropy (see Theorem \ref{t-6}) and therefore is not presented in the table.

In the next section we will illustrate our findings   by two examples.

\section{Examples}

We conclude this paper with a couple of simple illustrations. Consider the differential expression 
\[
l_\nu=-\frac{d^2}{dx^2}+\frac{\nu^2-1/4}{x^2},\;\; x\in [1,\infty)
\]
of order $\nu\ge \frac{1}{2}$ in the Hilbert space $\calH=L^2[1,\infty)$.
The minimal symmetric operator
\begin{equation}\label{ex-128}
 \left\{ \begin{array}{l}
 \dA\, y=-y^{\prime\prime}+\frac{\nu^2-1/4}{x^2}y \\
 y(1)=y^{\prime}(1)=0 \\
 \end{array} \right.
\end{equation}
 generated by this expression and boundary conditions has deficiency indices $(1,1)$ and is obviously nonnegative. Consider also the operator
\begin{equation}\label{ex-79}
 \left\{ \begin{array}{l}
 T_h\, y=-y^{\prime\prime}+\frac{\nu^2-1/4}{x^2}y \\
 y'(1)=h y(1). \\
 \end{array} \right.
\end{equation}

\subsection*{Example 1} Let $\nu=1/2$. It is known \cite{ABT} that in this case
$$
m_{\infty,\frac{1}{2}}(z)= -{i}{\sqrt{z}}\quad\textrm{ and hence}\quad m_{\infty,\frac{1}{2}}(i)=\frac{1}{\sqrt2}-\frac{1}{\sqrt2}i.
$$
The minimal symmetric operator then becomes
$$
 \left\{ \begin{array}{l}
 \dA\, y=-y^{\prime\prime} \\
 y(1)=y^{\prime}(1)=0. \\
 \end{array} \right.
$$
Note that in this case
$$\RE m_{\infty,\frac{1}{2}}(i)=\frac{1}{\sqrt2}> m_{\infty,\frac{1}{2}}(0)=0,$$
which taking into account that the symmetric operator $\dA$ is nonnegative,  illustrates the result of Lemma \ref{l-16}.

To construct a family of L-systems with von Neumann's parameter $\kappa=0$ for this case we set
$$
h=-m_{\infty,\frac{1}{2}}(i)=-\frac{1}{\sqrt2}+\frac{1}{\sqrt2}i.
$$
Then the main operator is given by
\begin{equation}\label{ex-80}
 \left\{ \begin{array}{l}
 T_{h_0}\, y=-y^{\prime\prime} \\
 y'(1)=\left(-\frac{1}{\sqrt2}+\frac{1}{\sqrt2}i\right) y(1) \\
 \end{array} \right.
\end{equation}
and it will be shared by all the family of L-systems with $\kappa=0$ and
\begin{equation}\label{e-140-h0}
h_0=-\frac{1}{\sqrt2}+\frac{1}{\sqrt2}i.
\end{equation}
Clearly, any L-system $\Theta_{\mu,h_0}$ has infinite c-entropy for all $\mu\in\Bbb R\cup\{\infty\}$.
The quasi-kernel of the real part of the state-space operator of this family of L-systems is determined by \eqref{e-31} as follows
\begin{equation}\label{ex-81}
 \left\{ \begin{array}{l}
\hat A_{\mu}\, y=-y^{\prime\prime} \\
 y'(1)=-\frac{\mu+\sqrt2}{\sqrt2 \mu+1} y(1) .\\
 \end{array} \right.
\end{equation}
To construct an L-system with an accretive state-space operator, we take  $\mu=\infty$ to have
$$
 \left\{ \begin{array}{l}
\hat A_{\infty}\, y=-y^{\prime\prime} \\
 y'(1)=-\frac{1}{\sqrt2} y(1) .\\
 \end{array} \right.
$$
The state-space operator of the L-system $\Theta_{\infty,h_0}$ with $\kappa=0$, $h_0$ defined by \eqref{e-140-h0}, and $\mu=\infty$ is (see \eqref{137})
\begin{equation}\label{ex-98}
\begin{split}
&\bA_{\infty, h_0}\, y=-y^{\prime\prime}-\frac{1}{\sqrt2}[{\sqrt2} y^{\prime}(1)+\left(1-i\right)y(1)]\,\delta (x-1), \\
&\bA^*_{\infty, h_0}\, y=-y^{\prime\prime}-\frac{1}{\sqrt2}[{\sqrt2}y^{\prime}(1)+\left(1+i\right)y(1)]\,\delta(x-1).
\end{split}
\end{equation}
Also the channel operator $K_{\infty, h_0}\,c=cg_{\infty, h_0}$, $(c\in \dC)$, where (see \eqref{146})
$$
g_{\infty, h_0}=2^{-\frac{1}{4}}\delta (x-1).
$$
Then $\Theta_{\infty, h_0}$ has the form
\begin{equation}\label{ex-99-system}
\Theta_{\infty, h_0}= \begin{pmatrix}
\bA_{\infty, h_0}&K_{\infty, h_0}&1\cr \calH_+ \subset
L_2[1,+\infty) \subset \calH_-& &\dC\cr \end{pmatrix},
\end{equation}
where all the components are described above. Using formulas \eqref{150} and \eqref{1501} we obtain
\begin{equation}\label{ex-99}
W_{\Theta_{\infty,h_0}}(z)=\frac{m_{\infty,\frac{1}{2}}(z)+ \overline h}{m_{\infty,\frac{1}{2}}(z)+h}=\frac{{i}{\sqrt{2z}} +1+i}{{i}{\sqrt{2z}}+1-i}
\end{equation}
and
\begin{equation}\label{ex-100}
V_{\Theta_{\infty,h_0}}(z)=\frac{\IM h}{m_{\infty,\frac{1}{2}}(z)+\RE h}=-\frac{1}{i\sqrt{2z}+1}.
\end{equation}
Direct substitution confirms that $V_{\Theta_{\infty,h_0}}(i)=i$. Clearly, $V_{\Theta_{\infty,h_0}}(z)\in\sM$ and our L-system $\Theta_{\infty,0}$ has infinite c-entropy according to Remark \ref{t-14}.

Similarly we can consider L-systems for which $h\ne -m_{\infty,\frac{1}{2}}(i)$ and hence $\kappa\ne0$. Then according to \eqref{e-41} we have
$$
\begin{aligned}
\kappa&=\left|\frac{m_{\infty,\frac{1}{2}}(i)+h}{m_{\infty,\frac{1}{2}}(i)+ \overline h}\right|=\left|\frac{\frac{1}{\sqrt2}-\frac{1}{\sqrt2}i+h}{\frac{1}{\sqrt2}-\frac{1}{\sqrt2}i+ \overline h}\right|=\left|\frac{1-i+\sqrt2\,\RE h+\sqrt2\,\IM h\, i}{1-i+ \sqrt2\,\RE h-\sqrt2\,\IM h\, i}\right|\\
&=\sqrt{\frac{(1+\sqrt2\,\RE h)^2+(\sqrt2\,\IM h-1)^2}{(1+\sqrt2\,\RE h)^2+(\sqrt2\,\IM h+1)^2}}.
\end{aligned}
$$
If we want our operator $T_h$ to be extremal, we (according to Theorem \ref{t-8}) set $\RE h=-m_{\infty,\frac{1}{2}}(-0)=0$. Then by the above calculations
$$
\kappa_{ext}=\sqrt{\frac{1+(\sqrt2\,\IM h-1)^2}{1+(\sqrt2\,\IM h+1)^2}}=\sqrt{\frac{1+(\IM h)^2-\sqrt2\,\IM h}{1+(\IM h)^2+\sqrt2\,\IM h}},
$$
where $\kappa_{ext}$ is the von Neumann parameter of the operator $T_h$.
As we have established in Section \ref{s5}, if the main operator $T_h$ is extremal, then $\kappa_{ext}$ satisfies $\kappa_0\le\kappa_{ext}<1$, where $\kappa_0$ is given by \eqref{e-47}. In our case $m_{\infty,\frac{1}{2}}(-0)=0$ and $|m_{\infty,\frac{1}{2}}(i)|=1$, and hence
\begin{equation}\label{e-104-kappa0}
    \kappa_0=\sqrt{\frac{\sqrt{2}-1}{\sqrt{2}+1}}=\sqrt2-1.
\end{equation}
Let us describe L-systems having the main extremal operator  with von Neumann's parameter $\kappa_0$ as in \eqref{e-104-kappa0}. The corresponding to $\kappa_0$ value of $h$ is (see \eqref{e-90-imh})
$$
h=-m_{\infty,\frac{1}{2}}(-0)+i\sqrt{(\RE m_\infty(i)-m_{\infty,\frac{1}{2}}(-0))^2+(\IM m_{\infty,\frac{1}{2}}(i))^2}=i|m_\infty(i)|=i.
$$
We have
\begin{equation}\label{ex-107}
 \left\{ \begin{array}{l}
 T_i\, y=-y^{\prime\prime} \\
 y'(1)=i y(1). \\
 \end{array} \right.
\end{equation}
The quasi-kernel of the real part of the state-space operator of our family of L-systems is determined by \eqref{e-31} as follows
\begin{equation}\label{ex-107-mu}
 \left\{ \begin{array}{l}
\hat A_{\mu}\, y=-y^{\prime\prime} \\
 y'(1)=-\frac{1}{\mu} y(1) .\\
 \end{array} \right.
\end{equation}

Now we construct two L-systems having $T_i$ as their main operator and attaining maximal c-entropy. We note that in the case in question the coefficient of dissipation $\calD=1$ and $\IM m_{\infty,\frac{1}{2}}(i)=-1/\sqrt2$. First we use \eqref{e-112-mu1} to obtain  $\mu_1=-1$. We get
$$
 \left\{ \begin{array}{l}
\hat A_{-1}\, y=-y^{\prime\prime} \\
 y'(1)=y(1) .\\
 \end{array} \right.
$$
The state-space operator of the L-system $\Theta_{-1,i}$ with $\kappa_0$ from \eqref{e-104-kappa0} and $\mu=-1$ is (see \eqref{137})
\begin{equation}\label{ex-108-A}
\begin{split}
&\bA_{-1,i}\, y=-y^{\prime\prime}-\frac {1}{1+i}\,[y^{\prime}(1)-iy(1)]\,[\delta(x-1)-\delta^{\prime}(x-1)], \\
&\bA^*_{-1,i}\, y=-y^{\prime\prime}-\frac {1}{1-i}\,[y^{\prime}(1)+iy(1)]\,[\delta(x-1)-\delta^{\prime}(x-1)].
\end{split}
\end{equation}
Also the channel operator is given by $K_{-1, i}\,c=cg_{-1, i}$, $(c\in \dC)$, where (see \eqref{146})
$$
g_{-1, i}=\sqrt2[\delta^{\prime}(x-1)-\delta (x-1)].
$$
Then $\Theta_{-1, i}$ has the form
\begin{equation}\label{ex-109-system}
\Theta_{-1, i}= \begin{pmatrix}
\bA_{-1,i}&K_{-1,i}&1\cr \calH_+ \subset
L_2[1,+\infty) \subset \calH_-& &\dC\cr \end{pmatrix},
\end{equation}
where all the entries are described above. Using formulas \eqref{150} and \eqref{1501} we obtain
\begin{equation}\label{ex-110}
W_{\Theta_{-1,i}}(z)=\frac{-1-i}{-1+i}\cdot\frac{m_{\infty,\frac{1}{2}}(z)+ \overline h}{m_{\infty,\frac{1}{2}}(z)+h}=i\frac{\sqrt{z} +1}{\sqrt{z}-1}
\end{equation}
and
\begin{equation}\label{ex-121}
V_{\Theta_{-1,i}}(z)=\frac{-i\sqrt{z}-1}{i\sqrt{z}-1}=-\frac{\sqrt{z}-i}{\sqrt{z}+i}.
\end{equation}
Direct substitution yields
$$V_{\Theta_{-1,i}}(i)=-\frac{\sqrt{i}-i}{\sqrt{i}+i}=-\frac{\frac{1}{\sqrt2}+\frac{1}{\sqrt2}i-i}{\frac{1}{\sqrt2}+\frac{1}{\sqrt2}i+i}=(\sqrt2-1)i.$$
Clearly, $V_{\Theta_{-1,i}}(z)\in\sM_{\kappa_0}$, where $\kappa_0$ is given by \eqref{e-104-kappa0}. The L-system $\Theta_{-1, i}$ in \eqref{ex-109-system} exemplifies the first L-system described in Remark \ref{r-30}.

In order to construct the second L-system described in Remark \ref{r-30} we apply \eqref{e-114-mu2} and obtain $\mu_2=1$ first. Then we get
$$
 \left\{ \begin{array}{l}
\hat A_{1}\, y=-y^{\prime\prime} \\
 y'(1)=-y(1) .\\
 \end{array} \right.
$$
The state-space operator of the L-system $\Theta_{1,i}$ with $\kappa_0$ from \eqref{e-104-kappa0} and $\mu=1$ is (see \eqref{137})
\begin{equation}\label{ex-112}
\begin{split}
&\bA_{1,i}\, y=-y^{\prime\prime}-\frac {1}{1-i}\,[y^{\prime}(1)-iy(1)]\,[\delta(x-1)+\delta^{\prime}(x-1)], \\
&\bA^*_{1,i}\, y=-y^{\prime\prime}-\frac {1}{1+i}\,[y^{\prime}(1)+iy(1)]\,[\delta(x-1)+\delta^{\prime}(x-1)].
\end{split}
\end{equation}
Also the channel operator is given by $K_{1, i}\,c=cg_{1, i}$, $(c\in \dC)$, where (see \eqref{146})
$$
g_{1, i}=\sqrt2[\delta(x-1)+\delta^{\prime}(x-1)].
$$
Then $\Theta_{1, i}$ has the form
\begin{equation}\label{ex-113-system}
\Theta_{1, i}= \begin{pmatrix}
\bA_{1,i}&K_{1,i}&1\cr \calH_+ \subset
L_2[1,+\infty) \subset \calH_-& &\dC\cr \end{pmatrix},
\end{equation}
where all the components are described above. Using formulas \eqref{150} and \eqref{1501} we obtain
\begin{equation}\label{ex-114}
W_{\Theta_{1,i}}(z)=\frac{1-i}{1+i}\cdot\frac{m_{\infty,\frac{1}{2}}(z)+ \overline h}{m_{\infty,\frac{1}{2}}(z)+h}=-i\frac{\sqrt{z} +1}{\sqrt{z}-1}
\end{equation}
and
\begin{equation}\label{ex-115}
V_{\Theta_{1,i}}(z)=\frac{\sqrt{z}+i}{\sqrt{z}-i}.
\end{equation}
Hence
$$V_{\Theta_{1,i}}(i)=\frac{\sqrt{i}+i}{\sqrt{i}-i}=\frac{\frac{1}{\sqrt2}+\frac{1}{\sqrt2}i+i}{\frac{1}{\sqrt2}+\frac{1}{\sqrt2}i-i}=(\sqrt2+1)i.$$
Clearly, $V_{\Theta_{1,i}}(z)\in\sM_{\kappa_0}^{-1}$, where $\kappa_0$ is given by \eqref{e-104-kappa0}. The L-system $\Theta_{1, i}$ in \eqref{ex-113-system} exemplifies the second L-system described in Remark \ref{r-30}.

\subsection*{Example 2} Let $\nu=3/2$. It is known \cite{ABT} that in this case
$$
m_{\infty,\frac{3}{2}}(z)= -\frac{iz-\frac{3}{2}\sqrt{z}-\frac{3}{2}i}{\sqrt{z}+i}-\frac{1}{2}=\frac{\sqrt{z}-iz+i}{\sqrt{z}+i}=1-\frac{iz}{\sqrt{z}+i}
$$
and
$$
m_{\infty,\frac{3}{2}}(-0)=1,\qquad m_{\infty,\frac{3}{2}}(i)=1+\frac{1}{\sqrt{2}}-\frac{i}{2}=\frac{2+\sqrt2-i}{2}.
$$
The minimal symmetric operator then becomes
$$
 \left\{ \begin{array}{l}
 \dA\, y=-y^{\prime\prime}+\frac{2}{x^2}y \\
 y(1)=y^{\prime}(1)=0. \\
 \end{array} \right.
$$
Note that in this case
$$\RE m_{\infty,\frac{1}{2}}(i)=1+\frac{1}{\sqrt2}> m_{\infty,\frac{1}{2}}(0)=1,$$
which taking into account that the symmetric operator $\dA$ is nonnegative,  illustrates the result of Lemma \ref{l-16}.

We are going to construct an L-system with an extremal main operator and maximum c-entropy. In order to do that we take $h$ defined by \eqref{e-109-h}. Then
\begin{equation}\label{e-134-h}
h=-1+i\sqrt{\left(1+\frac{1}{\sqrt2}-1\right)^2+\frac{1}{4}}=-1+\frac{\sqrt{3}}{2}i.
\end{equation}
Then the corresponding main operator $T_h$ is
\begin{equation}\label{ex-135}
 \left\{ \begin{array}{l}
 T_{h}\, y=-y^{\prime\prime}+\frac{2}{x^2}y \\
 y'(1)=\left(-1+\frac{\sqrt{3}}{2}i\right) y(1). \\
 \end{array} \right.
\end{equation}
This value of $h$ in \eqref{e-134-h} corresponds to the value of $\kappa_0$  given by \eqref{e-47} that is in this case
\begin{equation}\label{e-127-kappa0}
    \kappa_0=\frac{\sqrt2}{\sqrt{3}+1}.
\end{equation}
Having in mind the result of Theorem \ref{t-26}, we would like to  have an L-system with maximum c-entropy and accretive extremal state-space operator. To do so we take $\mu=\infty$ and construct the corresponding L-system $\Theta_{h,\infty}$. According to \eqref{e-31} the quasi-kernel in this case is
$$
 \left\{ \begin{array}{l}
\hat A_{\infty}\, y=-y^{\prime\prime} +\frac{2}{x^2}y\\
 y'(1)=- y(1),\\
 \end{array} \right.
$$
which is the Krein-von Neumann extension of $\dA$.
The state-space operator of this L-system $\Theta_{h,\infty}$ with $\kappa=\kappa_0$ and $\mu=\infty$ is (see \eqref{137})
\begin{equation*}\label{ex-136}
\begin{split}
&\bA_{\infty, h}\, y=-y^{\prime\prime}+\frac{2}{x^2}y-\left[ y^{\prime}(1)+\left(1-\frac{\sqrt{3}}{2}i\right)y(1)\right]\,\delta (x-1), \\
&\bA^*_{\infty, h}\, y=-y^{\prime\prime}+\frac{2}{x^2}y-\left[y^{\prime}(1)+\left(1+\frac{\sqrt{3}}{2}i\right) y(1)\right]\,\delta(x-1).
\end{split}
\end{equation*}
Its impedance function $V_{\Theta_{\infty,h}}(z)$ is such that
$$
V_{\Theta_{\infty,h}}(i)=\frac{\IM h}{m_{\infty,\frac{3}{2}}(i)+\RE h}=\frac{\frac{\sqrt{3}}{2}}{\frac{2+\sqrt2-i}{2}-1}=\frac{\sqrt3}{\sqrt2-i}=\sqrt{\frac{2}{3}}+\frac{1}{\sqrt{3}}i.
$$

Now we are going to  construct two L-systems having $T_h$ from \eqref{ex-135} as their main operator and such that their corresponding impedance functions are from the classes $\sM_{\kappa_0}$ and $\sM_{\kappa_0}^{-1}$ (as discussed in Remark \ref{r-30}), where $\kappa_0$ is given by \eqref{e-127-kappa0}.  First find the normalizing parameter $a$ that corresponds to $\kappa_0$,
$$
a=\frac{1-\kappa_0}{1+\kappa_0}=\frac{1-\frac{\sqrt2}{\sqrt{3}+1}}{1+\frac{\sqrt2}{\sqrt{3}+1}}=\frac{\sqrt{3}-\sqrt{2}+1}{\sqrt{3}+\sqrt{2}+1}=\sqrt3-\sqrt2.
$$
Then we use $V_{\Theta_{\mu_1,h}}(i)=ai$ and \eqref{150} for above value of $a$ to obtain
$$
V_{\Theta_{\mu_1,h}}(i)=\frac{\left(\frac{2+\sqrt2-i}{2}+\mu_1\right)\frac{\sqrt{3}}{2}}{\left(\mu_1+1\right)(\frac{2+\sqrt2-i}{2})-\mu_1-\frac{7}{4}} =(\sqrt3-\sqrt2)i.
$$
Solving the above equation for $\mu_1$ we get $\mu_1=-\frac{2+\sqrt3}{2}$. This yields the quasi-kernel (see \eqref{e-31})
$$
 \left\{ \begin{array}{l}
\hat A_{\mu_1}\, y=-y^{\prime\prime}+\frac{2}{x^2}y \\
 y'(1)=(\frac{\sqrt3}{2}-1)y(1) .\\
 \end{array} \right.
$$
The state-space operator of the L-system $\Theta_{\mu_1,h}$ with $\kappa_0$ from \eqref{e-104-kappa0} and \newline $\mu_1=-\frac{2+\sqrt3}{2}$ is (see \eqref{137})
\begin{equation*}\label{ex-136-A}
\begin{split}
&\bA_{\mu_1,h}\, y=-y^{\prime\prime}+\frac{2}{x^2}y-\frac{2y^{\prime}(1)+(1-{\sqrt3}i)y(1)}{{2}{\sqrt3}(1+i)}\left[{(2+\sqrt3)}\delta(x-1)-2\delta^{\prime}(x-1)\right], \\
&\bA^*_{\mu_1,h}\, y=-y^{\prime\prime}+\frac{2}{x^2}y-\frac{2y^{\prime}(1)+(1+{\sqrt3}i)y(1)}{{2}{\sqrt3}(1-i)}\left[{(2+\sqrt3)}\delta(x-1)-2\delta^{\prime}(x-1)\right].
\end{split}
\end{equation*}
Also the channel operator is given by $K_{\mu_1, h}\,c=cg_{\mu_1, h}$, $(c\in \dC)$, where (see \eqref{146})
$$
g_{\mu_1, h}=\frac{1}{2\cdot 3^{1/4}}\left[2\delta^{\prime}(x-1)-(2+\sqrt3)\delta (x-1)\right].
$$
Then $\Theta_{\mu_1, h}$ has the form
\begin{equation}\label{ex-137-system}
\Theta_{\mu_1, h}= \begin{pmatrix}
\bA_{\mu_1, h}&K_{\mu_1, h}&1\cr \calH_+ \subset
L_2[1,+\infty) \subset \calH_-& &\dC\cr \end{pmatrix},
\end{equation}
where all the components are described above. Using formulas \eqref{150} and \eqref{1501} we obtain
\begin{equation*}\label{ex-138}
W_{\Theta_{\mu_1, h}}(z)=\frac{-1-i}{-1+i}\cdot\frac{m_{\infty,\frac{3}{2}}(z)+ \overline h}{m_{\infty,\frac{3}{2}}(z)+h}=(-i)\frac{2z+\sqrt{3z} +\sqrt3 i}{2z-\sqrt{3z}-\sqrt3 i}
\end{equation*}
and
\begin{equation}\label{ex-139}
V_{\Theta_{\mu_1, h}}(z)=\frac{\left(m_{\infty,\frac{3}{2}}(z)+\mu_1\right)\IM h}{\left(\mu_1-\RE h\right)m_{\infty,\frac{3}{2}}(z)+\mu_1\RE h-|h|^2}=\frac{2iz-\sqrt{3z}-i\sqrt3}{2iz+\sqrt{3z}+i\sqrt3}.
\end{equation}
Direct substitution confirms that
$$V_{\Theta_{\mu_1,h}}(i)=\frac{2i\cdot i-\sqrt{3i}-i\sqrt3}{2i\cdot i+\sqrt{3i}+i\sqrt3}=(\sqrt3-\sqrt2)i.$$
Clearly, $$V_{\Theta_{\mu_1,h}}(z)\in\sM_{\kappa_0},$$ where $\kappa_0$ is given by \eqref{e-127-kappa0}.

Now we apply \eqref{e-114-mu2} and obtain
$$
\mu_2=\frac{\mu_1\RE h-|h|^2}{\mu_1-\RE h}=\frac{\frac{2+\sqrt3}{2}-\frac{7}{4}}{1-\frac{2+\sqrt3}{2}}=\frac{\sqrt3-2}{2}.
$$
This yields
$$
 \left\{ \begin{array}{l}
\hat A_{\mu_2}\, y=-y^{\prime\prime}+\frac{2}{x^2}y \\
 y'(1)=(\frac{\sqrt3}{2}+1)y(1) .\\
 \end{array} \right.
$$
The state-space operator of the L-system $\Theta_{\mu_2,h}$ with $\kappa_0$ from \eqref{e-104-kappa0} and $\mu_2=\frac{\sqrt3-2}{2}$ is (see \eqref{137})
\begin{equation*}\label{ex-140}
\begin{split}
&\bA_{\mu_2,h}\, y=-y^{\prime\prime}-\frac{2}{x^2}y+\frac{2y^{\prime}(1)+(1-{\sqrt3}i)y(1)}{{2}{\sqrt3}(1-i)}\left[{(\sqrt3-2)}\delta(x-1)-2\delta^{\prime}(x-1)\right], \\
&\bA^*_{\mu_2,h}\, y=-y^{\prime\prime}-\frac{2}{x^2}y+\frac{2y^{\prime}(1)+(1+{\sqrt3}i)y(1)}{{2}{\sqrt3}(1+i)}\left[{(\sqrt3-2)}\delta(x-1)-2\delta^{\prime}(x-1)\right].
\end{split}
\end{equation*}
Also the channel operator is given by $K_{\mu_2, h}\,c=cg_{\mu_2, h}$, $(c\in \dC)$, where (see \eqref{146})
$$
g_{\mu_2, h}=\frac{1}{2\cdot 3^{1/4}}\left[2\delta^{\prime}(x-1)+(\sqrt3-2)\delta (x-1)\right].
$$
Then the L-system $\Theta_{\mu_2, h}$ has the form
\begin{equation}\label{ex-141-system}
\Theta_{\mu_2, h}= \begin{pmatrix}
\bA_{\mu_2, h}&K_{\mu_2, h}&1\cr \calH_+ \subset
L_2[1,+\infty) \subset \calH_-& &\dC\cr \end{pmatrix}.
\end{equation}
 Using formulas \eqref{150}  we obtain
\begin{equation}\label{ex-142}
W_{\Theta_{\mu_2, h}}(z)=\frac{1-i}{1+i}\cdot\frac{m_{\infty,\frac{3}{2}}(z)+ \overline h}{m_{\infty,\frac{3}{2}}(z)+h}=(i)\frac{2z+\sqrt{3z} +\sqrt3 i}{2z-\sqrt{3z}-\sqrt3 i}
\end{equation}
and
\begin{equation}\label{ex-143}
V_{\Theta_{\mu_2, h}}(z)=-\frac{1}{V_{\Theta_{\mu_1, h}}(z)}=-\frac{2iz+\sqrt{3z}+i\sqrt3}{2iz-\sqrt{3z}-i\sqrt3}.
\end{equation}
So that
$$V_{\Theta_{\mu_2, h}}(i)=\frac{1}{\sqrt3-\sqrt2}i=(\sqrt3+\sqrt2)i,$$
and therefore $$V_{\Theta_{\mu_2, h}}(z)\in\sM_{\kappa_0}^{-1},$$ where $\kappa_0$ is given by \eqref{e-127-kappa0}.


\end{document}